\documentclass[12pt]{amsart}

\usepackage[cp1251]{inputenc}
\usepackage[T2A]{fontenc}
\usepackage[english]{babel}
\usepackage{amsfonts}
\usepackage{amsbsy}
\usepackage{amssymb}
\usepackage{amscd}
\usepackage{mathrsfs}
\usepackage{enumitem}
\usepackage{hyperref}

\renewcommand{\abovecaptionskip}{0pt}
\renewcommand{\belowcaptionskip}{6pt}
\makeatletter

\renewcommand{\@makecaption}[2]{
\vspace{\abovecaptionskip}%
\sbox{\@tempboxa}{#1. #2}%
\global\@minipagefalse \hbox to \hsize {{\scshape \hfil #1.
#2\hfil}} \vspace{\belowcaptionskip}}

\makeatother

\newcommand{\SL}{\operatorname{SL}}

\newcommand{\ad}{\mathrm{ad}}

\DeclareMathOperator{\Supp}{\mathrm{Supp}}
\DeclareMathOperator{\Hom}{\mathrm{Hom}}
\DeclareMathOperator{\ord}{\mathrm{ord}}

\DeclareMathOperator{\Spec}{\mathrm{Spec}}

\newcommand{\Dev}{\operatorname{Dev}}
\newcommand{\gr}{\operatorname{gr}}

\newcommand{\sat}{\mathrm{sat}}

\newcommand{\ZZ}{\mathbb Z}
\newcommand{\QQ}{\mathbb Q}

\newtheorem{theorem}{Theorem}

\newtheorem{proposition}[theorem]{Proposition}
\newtheorem{lemma}[theorem]{Lemma}
\newtheorem{corollary}[theorem]{Corollary}

\newtheorem*{question*}{Question}

\theoremstyle{definition}

\newtheorem{definition}[theorem]{Definition}
\newtheorem{example}[theorem]{Example}

\theoremstyle{remark}

\newtheorem{remark}[theorem]{Remark}

\makeatletter

\@addtoreset{theorem}{section}

\makeatother

\numberwithin{equation}{section}

\begin{document}

\renewcommand{\proofname}{Proof}
\renewcommand{\abstractname}{Abstract}
\renewcommand{\refname}{References}
\renewcommand{\figurename}{Figure}
\renewcommand{\tablename}{Table}

\title[On the irreducible components of moduli schemes]
{On the irreducible components of moduli schemes \\ for affine
spherical varieties}

\author{Roman Avdeev and St\'ephanie Cupit-Foutou}


\address{%
{\bf Roman Avdeev} \newline\indent National Research University ``Higher School of Economics'', Moscow, Russia}

\email{suselr@yandex.ru}

\address{%
{\bf St\'ephanie Cupit-Foutou}
\newline\indent Ruhr-Universit\"at Bochum, NA 4/67,
D-44797 Bochum, Germany}

\email{stephanie.cupit@rub.de}




\begin{abstract}
We give a combinatorial description of all affine spherical varieties with prescribed weight monoid~$\Gamma$.
As an application, we obtain a characterization of the irreducible components of Alexeev and Brion's moduli scheme $\mathrm M_\Gamma$ for such varieties. Moreover, we find several sufficient conditions for $\mathrm M_\Gamma$ to be irreducible and exhibit several examples where $\mathrm M_\Gamma$ is reducible. Finally, we provide examples of non-reduced~$\mathrm M_\Gamma$.
\end{abstract}

\maketitle

\section{Introduction}

Throughout this paper, we work over an algebraically closed field $\Bbbk$ of characteristic~$0$.

Let $G$ be a connected reductive algebraic group.
Fix a Borel subgroup $B$ of $G$ along with a maximal torus $T$ in $B$ and denote the related set of dominant weights by $\Lambda^+$.

Given a $G$-variety, that is, an algebraic variety $X$ equipped with a regular action of~$G$, the action of $G$ on $X$ naturally induces a $G$-module structure on the algebra $\Bbbk[X]$ of regular functions on~$X$. When $X$ is irreducible, the highest weights of $\Bbbk[X]$
form a monoid $\Gamma_X$ called the \textit{weight monoid} of~$X$. If furthermore $X$ is affine, its weight monoid is finitely generated.

An affine $G$-variety $X$ is said to be \textit{multiplicity-free} if $X$ is irreducible and the $G$-module $\Bbbk[X]$ contains every simple $G$-module with multiplicity at most~$1$. In this case, the $G$-module structure of $\Bbbk [X]$ is completely determined by the weight monoid of~$X$ and the decomposition of $\Bbbk [X]$ into simple $G$-modules reads as
\begin{equation} \label{eqn_direct_sum}
\Bbbk[X] = \bigoplus \limits_{\lambda \in \Gamma_X} \Bbbk[X]_\lambda,
\end{equation}
where $\Bbbk[X]_\lambda$ stands for the simple $G$-submodule of $\Bbbk[X]$ with highest weight~$\lambda$.

According to a result of Vinberg and Kimelfeld~\cite{VK78}, an irreducible affine $G$-variety is multiplicity-free if and only if it contains an open $B$-orbit. Normal irreducible (not necessarily affine) $G$-varieties containing an open $B$-orbit are said to be \textit{spherical}. In particular, an irreducible affine $G$-variety is spherical if and only if it is multiplicity-free and normal. For a multiplicity-free affine $G$-variety~$X$, the property of being normal (and hence spherical) can be read off from its weight monoid: $X$ is normal if and only if $\Gamma_X$ is \textit{saturated}, that is, equals the intersection of a lattice with a cone.

Given any finitely generated monoid $\Gamma \subset \Lambda^+$, there exists a multiplicity-free affine $G$-variety $X_0$ with weight monoid~$\Gamma$ for which the decomposition~(\ref{eqn_direct_sum}) is a grading, that is, $\Bbbk[X_0]_\lambda \Bbbk[X_0]_\mu = \Bbbk[X_0]_{\lambda+ \mu}$ for all $\lambda, \mu \in \Gamma$. As shown by Popov in \cite{Po86},
the $G$-variety $X_0$ is a common $G$-equivariant degeneration of all multiplicity-free affine $G$-varieties with weight monoid~$\Gamma$.

For a general multiplicity-free affine $G$-variety~$X$, the deviation of the decomposition~(\ref{eqn_direct_sum}) from being a grading is measured by the \textit{tail cone} of~$X$. This is the rational convex cone spanned by all expressions $\lambda+\mu - \nu$ such that $\lambda, \mu, \nu \in \Gamma_X$ and $\Bbbk[X]_\lambda \Bbbk[X]_\mu \supset \Bbbk[X]_\nu$. An invariant of importance to us is the set of \textit{spherical roots} of~$X$, which by definition are the primitive elements of the lattice spanned by $\Gamma_X$ lying on extremal rays of the tail cone of $X$. It is known that each spherical root of~$X$ is an element of a finite set $\Sigma(G)$ depending only on the group~$G$; see \S\,\ref{subsec_SHS_classification} for a description of~$\Sigma(G)$.

Losev proved in~\cite{Lo09b} that, up to a $G$-equivariant isomorphism, every affine spherical $G$-variety is uniquely determined by its weight monoid along with its set of spherical roots.
In \cite{ACF15}, this result was recovered by a different method and extended to arbitrary multiplicity-free affine $G$-varieties. It is therefore a natural problem to classify all multiplicity-free affine $G$-varieties with a prescribed weight monoid by determining all sets arising as sets of spherical roots of such varieties.

In this paper, we solve the above-mentioned problem in the case of affine spherical $G$-varieties. More precisely, for any given finitely generated and saturated monoid $\Gamma\subset\nobreak \Lambda^+$, we determine all possible sets $\Sigma$ (we call them \textit{admissible}) such that there exists an affine spherical $G$-variety with weight monoid $\Gamma$ and set of spherical roots~$\Sigma$ (see Theorem~\ref{Theorem-AS}). Our description is derived from the combinatorial classification of (not necessarily affine) spherical $G$-varieties established jointly in~\cite{LV, Kn91, Lu01, Lo09a, BraP14, Cu}.\footnote{See also the references in~\cite{BraP14} for partial results.} It appears that admissible sets are characterized by a number of combinatorial conditions, which can be easily checked in every concrete example.

From our description of the affine spherical $G$-varieties with prescribed weight monoid~$\Gamma$,
we derive a combinatorial characterization of the irreducible components of the moduli scheme $\mathrm M_\Gamma$ for these varieties that was constructed by Alexeev and Brion in~\cite{AB}. According to loc.~cit., for every finitely generated monoid $\Gamma \subset \Lambda^+$ (not necessarily saturated), $\mathrm M_\Gamma$ is an affine scheme of finite type equipped with an action of the adjoint torus $T_\ad$ (the quotient of $T$ by the center of~$G$) in such a way that $T_\ad$-orbits in $\mathrm M_\Gamma$ are in bijection with $G$-isomorphism classes of multiplicity-free affine $G$-varieties with weight monoid~$\Gamma$. Moreover, the variety $X_0$ may be regarded as the unique $T_\ad$-fixed closed point of~$\mathrm M_\Gamma$. It was also proved in~\cite{AB} (and recovered in~\cite{ACF15}) that $\mathrm M_\Gamma$ contains only finitely many $T_\ad$-orbits, so that there are only finitely many multiplicity-free affine $G$-varieties with any given weight monoid.

When $\Gamma$ is saturated, we show that the irreducible components of the moduli scheme $\mathrm M_\Gamma$ bijectively correspond to maximal with respect to inclusion admissible sets for~$\Gamma$ (see Theorem~\ref{indexationofIrrComp}). In particular, $\mathrm M_\Gamma$ is irreducible if and only if there is an admissible set that contains all the others. As an application of this criterion, we find a number of sufficient conditions on $\Gamma$ for $\mathrm M_\Gamma$ to be irreducible, two of these conditions read as follows:
\begin{enumerate}
\item \label{G-sat}
$\Gamma$ is $G$-saturated, that is, equals the intersection of a lattice with $\Lambda^+$ (see Theorem~\ref{thm_G-saturated});

\item \label{factorial}
$\Gamma$ is the weight monoid of an affine spherical $G$-variety whose algebra of regular functions is a unique factorization domain (see Proposition~\ref{prop_factorial_case}).
\end{enumerate}

Based on our description of the irreducible components of the moduli scheme~$\mathrm M_\Gamma$ (for saturated~$\Gamma$), we construct several examples of monoids $\Gamma$ such that $\mathrm M_\Gamma$ is reducible (see~\S\,\ref{subsec_examples_reducible}). The only example of that kind known before is due to D.~Luna (unpublished) and was mentioned in~\cite[Example~3.20]{AB2}.

Finally, combining our irreducibility criterion with results on the tangent space of $\mathrm M_\Gamma$ at the point~$X_0$ obtained in~\cite{ACF15}, we give a necessary and sufficient combinatorial condition for $\mathrm M_\Gamma$ to be an affine space (as a scheme); see Theorem~\ref{thm_smoothness_criterion}. Thanks to this condition, $\mathrm M_\Gamma$ turns out to be an affine space in both cases (\ref{G-sat}) and~(\ref{factorial}) mentioned above. As a particular case of~(\ref{factorial}), $\mathrm M_\Gamma$ is an affine space whenever $\Gamma$ is the weight monoid of a spherical $G$-module; this result was first proved by Papadakis and Van Steirteghem in \cite{PvS12,PvS16} via a case-by-case approach based on the classification of spherical modules. As another application of the above-mentioned combinatorial condition, we exhibit examples of monoids $\Gamma$ for which $\mathrm M_\Gamma$ is a non-reduced point (see~\S\,\ref{subsec_examples_non-reduced}).

We note that in~\cite{BvS} Bravi and Van Steirteghem independently obtained a similar combinatorial description of affine spherical $G$-varieties with a prescribed weight monoid. Their method is essentially the same as ours, however they used a slightly different language. The fact that $\mathrm M_\Gamma$ is irreducible in situation~(\ref{factorial}) was announced (without proof) by Pezzini in~\cite{Pe17}.

\subsection*{Organization of the paper}
In \S\,\ref{sec_notation} we set up the notation and conventions used throughout this paper. In \S\S\,\ref{sec_MF_var},\,\ref{sec_M_Gamma}, and~\ref{sec_spherical_varieties} we collect some basic material and known results on multiplicity-free affine $G$-varieties, Alexeev and Brion's moduli schemes~$\mathrm M_\Gamma$, and spherical $G$-varieties.
In \S\,\ref{sec_ASV_description} we obtain our combinatorial description of all affine spherical $G$-varieties with a prescribed weight monoid in terms of admissible sets. In \S\,\ref{sec_applications} we apply the results of \S\,\ref{sec_ASV_description} to characterize combinatorially the irreducible components of moduli schemes $\mathrm M_\Gamma$ and, in particular, to obtain an irreducibility criterion for~$\mathrm M_\Gamma$. Besides, we provide several conditions on $\Gamma$ under which $\mathrm M_\Gamma$ turns out to be irreducible or even an affine space. We end up \S\,\ref{sec_applications} by discussing several examples of $\mathrm M_\Gamma$ illustrating the diverse geometric properties of these schemes.

\subsection*{Acknowledgments}

The authors are grateful to E.\,B.~Vinberg for having attracted their attention to the problem of describing all affine spherical $G$-varieties with a prescribed weight monoid. Thanks are also due to the two referees for their valuable comments and suggestions, which improved the paper.

The first author was supported by the DFG priority program 1388, the ``Oberwolfach Leibniz Fellows'' programme of the Mathematical Research Institute of Oberwolfach, Dmitry Zimin's ``Dynasty'' Foundation, the Guest Program of the Max-Planck Institute for Mathematics in Bonn, and the RFBR grant no.~16-01-00818. He also thanks the Institute for Fundamental Science in Moscow for providing excellent working conditions.

\section{Notation and conventions}
\label{sec_notation}

Throughout this paper, all topological terms relate to the Zariski topology. All subgroups of algebraic groups are assumed to be closed. A \textit{variety} is a separated reduced scheme of finite type. A $K$-variety is a variety equipped with a regular action of an algebraic group~$K$. A $K$-isomorphism of two $K$-varieties is a $K$-equivariant isomorphism. Closed subsets of schemes are always equipped with their reduced subscheme structure.

$\ZZ^+ = \lbrace z \in \ZZ \mid z \ge 0 \rbrace$;

$\QQ^+ = \lbrace q \in \QQ \mid q \ge 0 \rbrace$;

$\Bbbk^\times$ is the multiplicative group of the field~$\Bbbk$;

$|X|$ is the cardinality of a finite set~$X$;

$V^*$ is the dual of a vector space~$V$;

$L^* = \Hom_\ZZ(L, \ZZ)$ is the dual lattice of a lattice~$L$;

$L_\QQ = L \otimes_\ZZ \QQ$ is the rational vector space spanned by a lattice~$L$;

$L^*_\QQ = (L_\QQ)^* = \Hom_\ZZ(L, \QQ)$

$\langle \cdot\,,\,\cdot \rangle$ is the natural pairing between $L^*_\QQ$ and~$L$, where $L$ is a lattice;

$\mathfrak X(K)$ is the character group of an algebraic group~$K$ (in additive notation);

$\overline Y$ is the closure of a subset $Y$ of a scheme~$X$;

$\Bbbk[X]$ is the algebra of regular functions on an algebraic variety~$X$;

$\mathcal O_X$ is the structure sheaf of a scheme~$X$;

$T_x X$ is the tangent space of a scheme $X$ at a closed point $x \in X$;

$G$ is a connected reductive algebraic group;

$C$ is the connected center of~$G$;

$B \subset G$ is a fixed Borel subgroup;

$T \subset B$ is a fixed maximal torus;

$U \subset B$ is the unipotent radical of~$B$;

$T_\ad$ is the adjoint torus of~$T$, that is, the quotient of $T$ by the center of~$G$;

$(\cdot\,,\,\cdot)$ is a fixed inner product on~$\mathfrak X(T)_\QQ$ invariant with respect to the Weyl group associated with~$T$;

$\Delta \subset \mathfrak X(T)$ is the root system of $G$ with respect to~$T$;

$\Pi \subset \Delta$ is the set of simple roots with respect to~$B$;

$\Delta^\vee \subset \mathfrak X(T)^*$ is the root system dual to~$\Delta$;

$\alpha^\vee \in \Delta^\vee$ is the coroot corresponding to a root $\alpha \in \Delta$;

$\Lambda^+ \subset \mathfrak X(T)$ is the monoid of dominant weights with respect to~$B$;

$V(\lambda)$ is the simple $G$-module with highest weight $\lambda \in \Lambda^+$;

$v_\lambda \in V(\lambda)$ is a highest weight vector in~$V(\lambda)$.

The lattices $\mathfrak X(B)$ and $\mathfrak X(T)$ are identified via restricting characters from $B$ to~$T$.

If $V$ is a vector space equipped with an action of a group~$K$,
then the notation $V^K$ stands for the subspace of $K$-invariant
vectors. For every character $\chi$ of~$K$, the notation $V^{(K)}_\chi$ stands for the subspace of $K$-semi-invariant vectors of weight~$\chi$.

The nodes of connected Dynkin diagrams as well as the simple roots of simple algebraic groups are numbered as in~\cite{Bo}.

For every element $\sigma = \sum \limits_{\alpha \in \Pi}k_\alpha \alpha$, where $k_\alpha \in \QQ^+$ for all $\alpha \in \Pi$, we set $\Supp \sigma = \lbrace \alpha \mid \nobreak k_\alpha \ne \nobreak 0 \rbrace$. The \textit{type} of $\sigma$ is the type of the Dynkin diagram of the set $\Supp \sigma$. When the Dynkin diagram of $\Supp \sigma$ is connected, we denote the $i$th simple root in $\Supp \sigma$ by~$\alpha_i$.\footnote{If the Dynkin diagram of $\Supp \sigma$ has non-trivial symmetries, this convention may not determine~$\alpha_i$ uniquely, however this does not cause any ambiguity in our paper.}

If the group $G$ is simple then $\varpi_i$ stands for the $i$th fundamental weight of~$G$.

For every subset $F \subset \mathfrak X(T)$, we set $F^\perp = \lbrace \alpha \in \Pi \mid \langle \alpha^\vee, \lambda \rangle = 0 \text{ for all } \lambda \in F \rbrace$. By abuse of notation, for a single element $\lambda \in \mathfrak X(T)$ we write $\lambda^\perp$ instead of $\lbrace \lambda \rbrace^\perp$.

Let $Q$ be a finite-dimensional vector space over~$\QQ$.

A subset $\mathcal C \subset Q$ is called a (finitely generated convex) \textit{cone} if there are finitely many elements $q_1, \ldots, q_s \in Q$ such that $\mathcal C = \QQ^+ q_1 + \ldots + \QQ^+ q_s$.

A cone $\mathcal C \subset \mathcal Q$ is said to be \textit{strictly convex} if $\mathcal C \cap (- \mathcal C) = \lbrace 0 \rbrace$.

The \textit{dimension} of a cone is the dimension of its linear span.

The \textit{dual cone} of a cone $\mathcal C \subset Q$ is the cone
\[
\mathcal C^\vee = \lbrace \xi \in Q^* \mid \langle \xi, q \rangle \ge 0 \text{\;for all\;} q \in \mathcal C \rbrace.
\]
One always has $(\mathcal C^\vee)^\vee = \mathcal C$.

A \textit{face} of a cone $\mathcal C \subset Q$ is a subset $\mathcal F \subset \mathcal C$ of the form
\[
\mathcal F = \mathcal C \cap \lbrace q \in Q \mid \langle \xi, q \rangle = 0 \rbrace
\]
for some $\xi \in \mathcal C^\vee$. Each face of $\mathcal C$ is again a cone.

An \textit{extremal ray} of a strictly convex cone $\mathcal C$ is a face of dimension~$1$.


\section{Generalities on multiplicity-free affine \texorpdfstring{$G$}{G}-varieties}
\label{sec_MF_var}

As already stated in the introduction, we say that an affine $G$-variety is multiplicity-free if it is irreducible and
its algebra of regular functions  is a multiplicity-free $G$-module. An affine spherical $G$-variety is a multiplicity-free affine $G$-variety that is normal.

\subsection{Combinatorial invariants}
\label{subsec_MF_comb_inv}

Let $X$ be a multiplicity-free affine $G$-variety.

The \textit{weight monoid of $X$} is the set $\Gamma_X$ of highest weights of the $G$-module $\Bbbk [X]$. Since $X$ is irreducible, $\Gamma_X$ is a submonoid of $\Lambda^+$. Moreover, $\Gamma_X$ is finitely generated, which is implied by the fact that $X$ is affine; see, for instance,~\cite[Corollary~5 of Theorem~4]{Po86}.

A monoid $\Gamma \subset \Lambda^+$ is said to be \textit{saturated} if $\Gamma = \ZZ\Gamma \cap \QQ^+ \Gamma$. It is well-known (and follows essentially from~\cite[\S\,1.2, Theorem~1]{Vu76} and~\cite[Ch.~I, \S\,1, Lemma~1]{KKMS73}) that $X$ is normal (and hence spherical) if and only if the weight monoid $\Gamma_X$ is saturated.

For every $\lambda \in \Gamma_X$, let $\Bbbk[X]_\lambda \subset \Bbbk[X]$ be the simple $G$-submodule with highest weight~$\lambda$.
An expression $\lambda + \mu - \nu$ with $\lambda, \mu,\nu \in \Gamma_X$ is called a \textit{tail} of $X$ if $\Bbbk[X]_\lambda \Bbbk[X]_\mu \supset \Bbbk[X]_\nu$.
The \textit{tail cone} of $X$ is the convex cone in $\QQ \Gamma_X$ generated by all tails.

The \textit{root monoid} of~$X$ is the monoid $\Xi_X$ generated by all tails. Note that $\Xi_X$ is a submonoid in $\ZZ^+ \Pi$.
Let $\Xi^\sat_X$ be the saturation of $\Xi_X$, that is, the intersection of the group $\ZZ \Xi_X$ with the tail cone.
A fundamental property of the monoid $\Xi^\sat_X$ is given by the following theorem, which is a particular case of~\cite[Theorem~1.3]{Kn96}.

\begin{theorem}
\label{thm_saturation_is_free}
The monoid $\Xi_X^\sat$ is free, and its indecomposable elements form a set of simple roots of a root system in~$\mathfrak X(T)$.
\end{theorem}

It follows from this theorem that the tail cone of $X$ is simplicial.
Let $\overline \Sigma_X$ denote the set of indecomposable elements of~$\Xi_X^\sat$. Clearly, the set $\overline \Sigma_X$ is linearly independent and
\[
\Xi_X^\sat = \ZZ^+ \overline \Sigma_X.
\]

A \textit{spherical root} of $X$ is a primitive element of the lattice $\ZZ \Gamma_X$ lying on an extremal ray of the tail cone of~$X$. We denote the set of spherical roots of~$X$ by~$\Sigma_X$.

The definitions of the sets $\Sigma_X$ and $\overline \Sigma_X$ imply that for every $\sigma \in \Sigma_X$ there is a unique element $\overline \sigma \in \overline \Sigma_X$ which is a multiple of~$\sigma$. Then the map $\sigma \mapsto \overline \sigma$ is a natural bijection between the sets $\Sigma_X$ and~$\overline \Sigma_X$. See Theorem~\ref{thm_overline_Sigma} for an explicit description of this bijection for affine spherical $G$-varieties.

In what follows, we shall need the following consequence of Theorem~\ref{thm_saturation_is_free}.

\begin{corollary}\label{crl_sr_are_li}
The set $\Sigma_X$ is linearly independent.
\end{corollary}

A refined version of Theorem~\ref{thm_saturation_is_free} for affine spherical $G$-varieties is given by

\begin{theorem}[{\cite[Theorem~4.12]{ACF15}}]
\label{thm_root_monoid_is_free}
Suppose that $X$ is an affine spherical $G$-variety. Then $\Xi_X = \Xi_X^\sat$; in particular, $\Xi_X$ is free.
\end{theorem}

The following uniqueness result was first proved by Losev in case of affine spherical $G$-varieties; see~\cite[Theorem~1.2]{Lo09b}.

\begin{theorem}[{\cite[Corollary~4.23]{ACF15}}] \label{thm_uniqueness_MF}
Up to a $G$-equivariant isomorphism, a multiplicity-free affine $G$-variety $X$ is uniquely determined by the pair $(\Gamma_X, \Sigma_X)$.
\end{theorem}

\subsection{The variety \texorpdfstring{$X_0$}{X0}}
\label{subsec_X_0}

Given a finitely generated submonoid $\Gamma \subset \Lambda^+$, take a finite generating set $\mathrm E$ of $\Gamma$ and consider the $G$-module
\[
V = \bigoplus \limits_{\lambda \in \mathrm E}
V(\lambda)^*.
\]
For every $\lambda \in \mathrm E$, choose a highest weight vector $v_{\lambda^*} \in V(\lambda)^*$, consider the vector
\[
x_0 = \sum \limits_{\lambda \in \mathrm E} v_{\lambda^*} \in V,
\]
and put
\[
X_0 = \overline{Gx_0} \subset V.
\]

\begin{theorem}[{\cite[Theorem~6]{VP72}}]
\label{thm_VP}
The following assertions hold:
\begin{enumerate}[label=\textup{(\alph*)},ref=\textup{\alph*}]
\item
up to a $G$-isomorphism, the $G$-variety $X_0$ is independent of the choice of~$\mathrm E$;

\item
$X_0$ is a multiplicity-free affine $G$-variety;

\item
$\Gamma_{X_0} = \Gamma$;

\item \label{thm_VP_d}
$\Bbbk[X_0]_\lambda \cdot \Bbbk[X_0]_\mu = \Bbbk[X_0]_{\lambda + \mu}$ for all $\lambda, \mu \in \Gamma_{X_0}$; in other words, $\Sigma_{X_0} = \varnothing$.
\end{enumerate}
\end{theorem}

\subsection{Degenerations}
Let $X$ be a multiplicity-free affine $G$-variety. We say that an element $\varrho \in \mathfrak X(T)^*$ is \textit{non-negative} (with respect to~$X$) if $\langle \varrho, \gamma \rangle \ge 0$ for all $\gamma \in \Gamma_X \cup \Sigma_X$. Note that every element of $\mathfrak X(T)^*$ lying in the dominant Weyl chamber of the root system~$\Delta^\vee$ is non-negative.

Take a non-negative element $\varrho \in \mathfrak X(T)^*$ and
for every $n \in \ZZ^+$ define the subspace
\[
D_{\varrho, n} = \bigoplus \limits_{\substack{\lambda \in \Gamma_X, \\ \langle \varrho, \lambda \rangle \le n}} \Bbbk[X]_\lambda \subset \Bbbk[X].
\]
Then the collection of subspaces $\lbrace D_{\varrho, n} \mid n \in \ZZ^+ \rbrace$ forms a $G$-invariant filtration on~$\Bbbk[X]$. Let
\[
\gr_\varrho \Bbbk[X] = \bigoplus \limits_{n \in \ZZ^+} D_{\varrho, n} / D_{\varrho, n-1}
\]
be the graded algebra associated with this filtration. Clearly, the algebras $\Bbbk[X]^U$ and $(\gr_\varrho \Bbbk[X])^U$ are isomorphic, which by~\cite[Corollary of Theorem~6]{Po86} implies that the algebra $\gr_\varrho \Bbbk[X]$ is a finitely generated integral domain. We now consider the irreducible affine $G$-variety
\[
\gr_\varrho X = \Spec (\gr_\varrho \Bbbk[X]).
\]

The following result follows directly from the construction.

\begin{lemma} \label{lemma_contraction}
The affine $G$-variety $Y = \gr_\varrho X$ is multiplicity-free. Moreover, $\Gamma_Y =\nobreak \Gamma_X$ and $\Sigma_Y = \lbrace \sigma \in \Sigma_X \mid \langle \varrho, \sigma \rangle = 0 \rbrace$. In particular, $\Sigma_Y \subset \Sigma_X$.
\end{lemma}

Note that if $\langle \varrho, \alpha \rangle > 0$ for all $\alpha \in \Pi$ then $\gr_\varrho X$ is $G$-isomorphic to the $G$-variety $X_0$ introduced in~\S\,\ref{subsec_X_0}; see~\cite[\S\,4]{Po86}.

\begin{proposition} \label{prop_MF_subset}
Suppose that $X$ is a multiplicity-free affine $G$-variety and  $\Sigma \subset \Sigma_X$ is an arbitrary subset. Then
\begin{enumerate}[label=\textup{(\alph*)},ref=\textup{\alph*}]
\item \label{prop_MF_subset_a}
there exists a non-negative element $\varrho \in \mathfrak X(T)^*$ such that $\langle \varrho, \sigma \rangle = 0$ for all $\sigma \in \Sigma$ and $\langle \varrho, \sigma \rangle > 0$ for all $\sigma \in \Sigma_X \setminus \Sigma$;

\item \label{prop_MF_subset_b}
for any $\varrho$ as in \textup(\ref{prop_MF_subset_a}\textup), the $G$-variety $Y = \gr_\varrho X$ satisfies $\Gamma_Y = \Gamma_X$ and $\Sigma_Y = \Sigma$.
\end{enumerate}
\end{proposition}

\begin{proof}
(\ref{prop_MF_subset_a}) Recall from Theorem~\ref{thm_saturation_is_free} that the elements in $\overline \Sigma_X$ form a set of simple roots of a root system in~$\mathfrak X(T)$. For every $\sigma \in \Sigma_X$, let $\varpi(\sigma) \in \mathfrak X(T)^*_\QQ$ be the fundamental coweight of this root system corresponding to the simple root~$\overline \sigma$, so that $\langle \varpi(\sigma), \overline \sigma \rangle = 1$ and $\langle \varpi(\sigma), \overline\sigma' \rangle = 0$ for all $\sigma' \in \Sigma_X \setminus \lbrace \sigma \rbrace$. Since every fundamental coweight of a root system is a non-negative linear combination of simple coroots and $\overline \Sigma_X \subset \ZZ^+\Pi$, it follows that for every $\sigma \in \Sigma_X$ the element $\varpi(\sigma)$ lies in $\QQ^+\lbrace \alpha^\vee \mid \alpha \in \Pi \rbrace$ and hence satisfies $\langle \varpi(\sigma), \lambda \rangle \ge 0$ for all $\lambda \in \Lambda^+$. Consequently, a suitable positive multiple of the element $\sum \limits_{\sigma \in \Sigma_X \setminus \Sigma} \varpi(\sigma)$ belongs to $\mathfrak X(T)^*$ and is non-negative, hence it can be taken for~$\varrho$.

(\ref{prop_MF_subset_b}) This is a consequence of part~(\ref{prop_MF_subset_a}) and Lemma~\ref{lemma_contraction}.
\end{proof}


\section{Generalities on moduli schemes \texorpdfstring{$\mathrm M_\Gamma$}{MGamma}}
\label{sec_M_Gamma}

Throughout this section we assume that $\Gamma \subset \Lambda^+$ is an arbitrary finitely generated monoid.

\subsection{The definition of~\texorpdfstring{$\mathrm M_\Gamma$}{M_Gamma}}
\label{subsec_def_of_M_Gamma}
Consider the $G$-module
\begin{equation} \label{eqn_A_Gamma}
A_\Gamma = \bigoplus_{\lambda \in \Gamma} V(\lambda).
\end{equation}
Fix a highest weight vector $v_\lambda \in V(\lambda)$ and equip the subspace $A_\Gamma^U = \bigoplus \limits_{\lambda \in \Gamma} \Bbbk v_\lambda \subset A_\Gamma$ with an algebra structure by setting
\begin{equation} \label{eqn_multiplication}
v_\lambda \cdot v_\mu = v_{\lambda + \mu} \quad \text{for all $\lambda, \mu \in \Gamma$}.
\end{equation}
Note that the algebra $A_\Gamma^U$ is isomorphic to the semigroup algebra of~$\Gamma$.

Let
\[
\mathcal M_\Gamma \colon \text{(Schemes)} \to \text{(Sets)}
\]
be the contravariant functor assigning to each scheme $S$ the set of $\mathcal O_S$-$G$-algebra structures on the sheaf $\mathcal O_S \otimes_\Bbbk A_\Gamma$ that extend the given multiplication~(\ref{eqn_multiplication}) on~$A_\Gamma^U$.

As a consequence of \cite[Proposition~2.10 and Theorems~1.12,~2.7]{AB} (see also \cite[\S\,4.3]{Br13}), the functor $\mathcal M_\Gamma$ is represented by an affine scheme $\mathrm M_\Gamma$ of finite type, called the \textit{moduli scheme of multiplicity-free affine $G$-varieties with weight monoid~$\Gamma$}. In particular, the closed points of $\mathrm M_\Gamma$ are in bijection with the $G$-equivariant algebra structures on $A_\Gamma$ extending the multiplication~(\ref{eqn_multiplication}) on~$A_\Gamma^U$.

Thanks to \cite[Theorem~2]{Po86}, for every multiplicity-free affine $G$-variety $X$ with weight monoid~$\Gamma$ there is a (not necessarily unique) $T$-equivariant isomorphism
\begin{equation} \label{eqn_isomorphism}
\tau \colon \Bbbk[X]^U \xrightarrow{\sim} A_\Gamma^U,
\end{equation}
which uniquely extends to a $G$-module isomorphism $\Bbbk[X] \xrightarrow{\sim} A_\Gamma$. The algebra structure on $A_\Gamma$ transferred from $\Bbbk[X]$ via this isomorphism thus determines a closed point of~$\mathrm M_\Gamma$. In this way, $X$ may be regarded as a closed point of $\mathrm M_\Gamma$ (which depends however on the choice of~$\tau$).

\subsection{Basic facts on the \texorpdfstring{$T_\ad$}{T_ad}-action on~\texorpdfstring{$\mathrm M_\Gamma$}{M_Gamma}}

The moduli scheme $\mathrm M_\Gamma$ can be equipped with an action of the adjoint torus $T_\ad$; see~\cite[\S\,2.1]{AB} for a precise definition. For convenience of the reader, we recall this action on the level of closed points. As was mentioned in \S\,\ref{subsec_def_of_M_Gamma}, each closed point of $\mathrm M_\Gamma$ is given by a multiplication law $m \colon A_\Gamma \otimes A_\Gamma \to A_\Gamma$ extending the multiplication~(\ref{eqn_multiplication}) on~$A_\Gamma^U$. It is clear from~(\ref{eqn_A_Gamma}) that $m$ can be expressed as the sum
\[
m = \sum \limits_{\lambda,\mu,\nu \in \Gamma} m_{\lambda,\mu}^\nu
\]
where each component $m_{\lambda,\mu}^\nu \colon V(\lambda) \otimes V(\mu) \to V(\nu)$ is a $G$-module homomorphism. Then \cite[Proposition~2.11]{AB} asserts that
\[
(t \cdot m)_{\lambda, \mu}^\nu = (\nu - \lambda - \mu)(t) \cdot m_{\lambda,\mu}^\nu
\]
for all $t \in T_\ad$ and $\lambda, \mu, \nu \in \Gamma$. It is worth noting that $T_\ad$ acts on the closed points of $\mathrm M_\Gamma$ just by changing the isomorphism $\tau$ in~(\ref{eqn_isomorphism}).

Below we gather several properties of the $T_\ad$-action on~$\mathrm M_\Gamma$.

\begin{theorem}[{see \cite[Theorem 1.12 and Lemma 2.2]{AB}}]
\label{thm_bijection_isoclasses}
Let $X$ be a multiplicity-free affine $G$-variety with weight monoid~$\Gamma$. Regard $X$ as a closed point of $\mathrm M_\Gamma$ via an isomorphism~$\tau$ as in~\textup{(\ref{eqn_isomorphism})}.
\begin{enumerate}[label=\textup{(\alph*)},ref=\textup{\alph*}]
\item
The $T_\ad$-orbit $T_\ad X \subset \mathrm M_\Gamma$ does not depend on the choice of~$\tau$.

\item \label{thm_bijection_isoclasses_b}
The map $X \mapsto T_\ad X$ induces a bijection between the $G$-isomorphism classes of multiplicity-free affine $G$-varieties with weight monoid~$\Gamma$ and the $T_\ad$-orbits in~$\mathrm M_\Gamma$.
\end{enumerate}
\end{theorem}

Recall the definition of the variety $X_0$ from~\S\,\ref{subsec_X_0}.

\begin{theorem}[{\cite[Theorem~2.7]{AB}}] \label{thm_X_0_fixed_point}
Regarded as a closed point of~$\mathrm M_\Gamma$, $X_0$ is fixed by~$T_\ad$ and belongs to each $T_\ad$-orbit closure in~$\mathrm M_\Gamma$.
\end{theorem}

The following result was first proved in~\cite[Corollary~3.4]{AB} and recovered by another method in~\cite[Corollary~4.24]{ACF15}.

\begin{theorem} \label{thm_finiteness_result}
The torus $T_\ad$ acts on $\mathrm M_\Gamma$ with finitely many orbits. In particular, there are only finitely many isomorphism classes of multiplicity-free affine $G$-varieties with prescribed weight monoid~$\Gamma$.
\end{theorem}

\begin{corollary} \label{crl_irr_comp}
The irreducible components of $\mathrm M_\Gamma$ are given by the closures of the open $T_\ad$-orbits in $\mathrm M_\Gamma$.
\end{corollary}

The next theorem provides a moduli interpretation of the root monoid of a multiplicity-free affine $G$-variety.

\begin{theorem}[{\cite[Proposition~2.13]{AB}}]
\label{thm_orbit_closure_of_X}
Let $X$ be a multiplicity-free affine $G$-variety with weight monoid~$\Gamma$. The $T_\ad$-orbit closure $\overline{T_\ad X}\subset \mathrm M_\Gamma$ is a multiplicity-free affine $T_\ad$-variety with weight monoid~$\Xi_X$.
\end{theorem}

\begin{corollary} \label{crl_dimension_of_CX}
Under the hypotheses of Theorem~\textup{\ref{thm_orbit_closure_of_X}}, $\dim T_\ad X = |\Sigma_X|$.
\end{corollary}

\subsection{Further properties}

The three equivalent conditions of the following proposition naturally define a partial order on the set of ($G$-isomorphism classes of) multiplicity-free affine $G$-varieties with weight monoid~$\Gamma$.

\begin{proposition} \label{prop_partial_order}
Let $X$ and $Y$ be multiplicity-free affine $G$-varieties with weight monoid~$\Gamma$. The following conditions are equivalent.
\begin{enumerate}[label=\textup{(\arabic*)},ref=\textup{\arabic*}]
\item \label{prop_partial_order_1}
$\overline{T_\ad Y} \subset \overline{T_\ad X}$.

\item \label{prop_partial_order_2}
$\Sigma_Y \subset \Sigma_X$.

\item \label{prop_partial_order_3}
$Y = \gr_\varrho X$ for some non-negative element $\varrho \in \mathfrak X(T)^*$.
\end{enumerate}
\end{proposition}

\begin{proof}
(\ref{prop_partial_order_1})$\Rightarrow$(\ref{prop_partial_order_2}) In view of Theorem~\ref{thm_orbit_closure_of_X} and well-known facts from the theory of (possibly non-normal) affine toric varieties (see, for instance, \cite[Theorem 3.A.3]{CLS}), the relation $\overline{T_\ad Y} \subset \overline{T_\ad X}$ implies that the cone $\QQ^+ \Xi_Y$ is a face of the cone $\QQ^+ \Xi_X$, which yields $\Sigma_Y \subset \Sigma_X$.

(\ref{prop_partial_order_2})$\Rightarrow$(\ref{prop_partial_order_1}) As the set $\Sigma_X$ is linearly independent (Corollary~\ref{crl_sr_are_li}), the cone $\QQ^+ \Sigma_Y$ is a face of the cone $\QQ^+ \Sigma_X$. It then follows from loc.~cit. that $\overline{T_\ad X}$ contains a $T_\ad$-orbit $O$ such that the weight monoid of $\overline O$ equals $\Xi_X \cap \QQ^+ \Sigma_Y$. By Theorem~\ref{thm_bijection_isoclasses}, there exists a multiplicity-free affine $G$-variety $Y'$ with weight monoid $\Gamma$ such that $O = T_\ad Y'$. Then $\Sigma_{Y'} = \Sigma_Y$, therefore $Y$ and $Y'$ are $G$-isomorphic by Theorem~\ref{thm_uniqueness_MF}.

(\ref{prop_partial_order_2})$\Rightarrow$(\ref{prop_partial_order_3})
This follows from Proposition~\ref{prop_MF_subset} and Theorem~\ref{thm_uniqueness_MF}.

(\ref{prop_partial_order_3})$\Rightarrow$(\ref{prop_partial_order_2})
This follows from Lemma~\ref{lemma_contraction}.
\end{proof}

The following smoothness criterion for $\mathrm M_\Gamma$ is known to specialists; we provide it together with a proof for convenience of the reader.

\begin{theorem} \label{thm_smoothness}
The following properties are equivalent.
\begin{enumerate}[label=\textup{(\arabic*)},ref=\textup{\arabic*}]
\item \label{thm_smoothness_1}
$\mathrm M_\Gamma$ is smooth at $X_0$.

\item \label{thm_smoothness_2}
$\mathrm M_\Gamma$ is smooth.

\item \label{thm_smoothness_3}
$\mathrm M_\Gamma$ is an affine space.
\end{enumerate}
\end{theorem}

\begin{proof}
(\ref{thm_smoothness_1})$\Rightarrow$(\ref{thm_smoothness_2})
Clearly, the set of singular points in $\mathrm M_\Gamma$ is closed and $T_\ad$-stable. If it is nonempty then it contains $X_0$ by Theorem~\ref{thm_X_0_fixed_point},
a~contradiction.

(\ref{thm_smoothness_2})$\Rightarrow$(\ref{thm_smoothness_3})
It follows from Theorem~\ref{thm_orbit_closure_of_X} that $\mathrm M_\Gamma$ is a smooth affine multiplicity-free $T_\ad$-variety. By Theorem~\ref{thm_X_0_fixed_point}, the closed $T_\ad$-orbit in $\mathrm M_\Gamma$ is the point~$X_0$, whence $\mathrm M_\Gamma$ is an affine space.

(\ref{thm_smoothness_3})$\Rightarrow$(\ref{thm_smoothness_1})
This implication is obvious.
\end{proof}

\section{Generalities on spherical varieties}
\label{sec_spherical_varieties}

As was already mentioned in the introduction, a $G$-variety $X$ is said to be spherical if it is normal, irreducible, and contains an open $B$-orbit. Thanks to \cite[Theorem~2]{VK78}, in the case of affine $X$ this definition agrees with that given at the beginning of~\S\,\ref{sec_MF_var}.

\subsection{Combinatorial invariants}
\label{subsec_comb_inv}

Let $X$ be an arbitrary spherical $G$-variety and let $\Bbbk(X)$ denote the field of rational functions on~$X$.

The \textit{weight lattice} of~$X$ is the set
\[
\Lambda_X = \lbrace \lambda \in \mathfrak X(T) \mid
\Bbbk(X)^{(B)}_\lambda \ne 0 \rbrace.
\]
Clearly, $\Lambda_X$ is a sublattice of $\mathfrak X(T)$. For every $\lambda \in \Lambda_X$, we fix a nonzero function $f_\lambda \in \Bbbk(X)^{(B)}_\lambda$. Since $X$ contains an open $B$-orbit, one has $\Bbbk(X)^{(B)}_\lambda = \Bbbk f_\lambda$ for all $\lambda \in \Lambda_X$.

We set
\[
\mathcal L_X = \Lambda_X^* \quad\mbox{ and }\quad \mathcal Q_X = (\mathcal L_X)_\QQ = (\Lambda_X)^*_\QQ.
\]
We regard $\mathcal L_X$ as a sublattice in~$\mathcal Q_X$.

Every discrete $\QQ$-valued valuation $v$ of $\Bbbk(X)$ vanishing on~$\Bbbk^\times$ determines the element
$\rho_v \in \mathcal Q_X$ such that $ \langle \rho_v, \lambda
\rangle = v(f_\lambda)$ for all $\lambda \in \Lambda_X$. The
restriction of the map $v \mapsto \rho_v$ to the set of
$G$-invariant $\QQ$-valued valuations of $\Bbbk(X)$ vanishing
on~$\Bbbk^\times$ is injective (see~\cite[7.4]{LV}
or~\cite[Corollary~1.8]{Kn91}); we denote its image by~$\mathcal
V_X$. Moreover, $\mathcal V_X \subset \mathcal Q_X$ is a
finitely generated convex cone of full dimension; see~\cite[4.1, Corollary,~i)]{BriP}
or~\cite[Corollary~5.3]{Kn91}. The cone $\mathcal V_X$ is called
the \textit{valuation cone} of~$X$.

Primitive elements $\sigma \in \Lambda_X$ such that $\QQ^+ \sigma$ is an extremal ray of the cone $-\mathcal V_X^\vee$ are called \textit{spherical roots} of~$X$. We denote the set of all spherical roots of~$X$ by~$\Sigma_X$.

From~\cite[\S\,3]{Br90} or~\cite[Theorem~7.4]{Kn94}, we know that $\Sigma_X$ is a set of simple roots of a root system in $\Lambda_X$. Hence $(\sigma_1, \sigma_2) \le 0$ for any two distinct elements $\sigma_1, \sigma_2 \in \Sigma_X$. In particular, the set $\Sigma_X$ is linearly independent.

Let $\mathcal B_X$ (resp. $\mathcal D_X$) denote the set of all $G$-stable (resp. $B$-stable but not $G$-stable) prime divisors in~$X$. Elements of $\mathcal D_X$ are called \textit{colors} of~$X$. Clearly, the union $\mathcal B_X \cup \mathcal D_X$ is the set of all $B$-stable prime divisors in~$X$. As $X$ contains an open $B$-orbit, the set $\mathcal B_X \cup \mathcal D_X$ is finite.

For every $D \in \mathcal B_X \cup \mathcal D_X$, let $v_D$ be the valuation of the field $\Bbbk(X)$ defined by~$D$, that is, $v_D(f) = \ord_D(f)$ for every $f \in \Bbbk(X)$. We define the map
\[
\rho_X \colon \mathcal B_X \cup \mathcal D_X \to \mathcal L_X
\]
by setting $\rho_X(D) = \rho_{v_D}$.

For every $\alpha \in \Pi$, let $\mathcal D_X(\alpha) \subset \mathcal D_X$ be the set of colors that are unstable with respect to the action of the minimal parabolic subgroup $P_\alpha\supset B$ of $G$ associated with~$\alpha$. Then the set $\mathcal D_X$ is the union of the sets $\mathcal D_X(\alpha)$ with $\alpha$ running over $\Pi$. We set
\begin{equation} \label{eqn_type_(p)}
\Pi^p_X = \lbrace \alpha \in \Pi \mid \mathcal D_X(\alpha) =
\varnothing \rbrace.
\end{equation}

\begin{remark} \label{rem_invariants}
It follows from the above definitions that the invariants $\Lambda_X$, $\mathcal L_X$, $\mathcal Q_X$, $\mathcal V_X$, $\Sigma_X$, $\mathcal D_X$, $\left.\rho_X\right|_{\mathcal D_X}$, and $\Pi^p_X$ depend only on the open $G$-orbit $O \subset X$.
The sets $\mathcal D_X$ and $\mathcal D_O$ are identified by intersecting colors of~$X$ with~$O$.
\end{remark}

\begin{remark}
If $X$ is affine then by~\cite[Lemma~5.1]{Kn91} the dual cone of $- \mathcal V_X$ is exactly the tail cone of~$X$ defined in~\S\,\ref{subsec_MF_comb_inv}. Taking into account Proposition~\ref{prop_weight_monoid}(\ref{prop_weight_monoid_b}) below, we see that in this case the set $\Sigma_X$ is exactly the set of primitive elements of $\Lambda_X$ lying on extremal rays of the tail cone of~$X$, which agrees with the definition of~$\Sigma_X$ given in~\S\,\ref{subsec_MF_comb_inv}.
\end{remark}

The following proposition, which is known to specialists, relates some of the above invariants of an affine spherical $G$-variety $X$ with its weight monoid $\Gamma_X$ introduced in~\S\,\ref{subsec_MF_comb_inv}.

\begin{proposition} \label{prop_weight_monoid}
Let $X$ be an affine spherical $G$-variety and let $\mathcal K_X$ be the cone in $\mathcal Q_X$ generated by the set $\rho_X(\mathcal B_X \cup \mathcal D_X)$.
\begin{enumerate}[label=\textup{(\alph*)},ref=\textup{\alph*}]
\item \label{prop_weight_monoid_a}
$\Gamma_X = \Lambda_X \cap \mathcal K_X^\vee$, where $\mathcal K_X^\vee$ is considered as a cone in $(\Lambda_X)_\QQ$.

\item \label{prop_weight_monoid_b}
$\Lambda_X = \ZZ \Gamma_X$.

\item \label{prop_weight_monoid_c}
$\Pi^p_X = \Gamma_X^\perp$.
\end{enumerate}
\end{proposition}

\begin{proof}
(\ref{prop_weight_monoid_a}) Let $\lambda \in \Lambda_X$. Since the function $f_\lambda$ is $B$-semi-invariant, it can have poles only in the complement of the open $B$-orbit in~$X$. The normality of $X$ implies that $f_\lambda$ is regular on~$X$ if and only if it has no poles along each of the divisors in $\mathcal B_X \cup \mathcal D_X$, which is equivalent to $\lambda \in \mathcal K^\vee_X$.

(\ref{prop_weight_monoid_b}) See, for instance, \cite[Proposition~5.14]{Tim}.

(\ref{prop_weight_monoid_c}) To prove the inclusion ``$\subset$'', let $\alpha \in \Pi^p_X$ and assume that $\langle \alpha^\vee, \lambda \rangle > 0$ for some $\lambda \in \Gamma_X$. Then the line $\Bbbk f_\lambda = \Bbbk[X]^{(B)}_\lambda$ is $P_\alpha$-unstable, hence so is the divisor of zeros of~$f_\lambda$ by~\cite[Theorem~3.1]{VP}. Consequently, $\mathcal D_X(\alpha) \ne \varnothing$, which contradicts~(\ref{eqn_type_(p)}).

Now let us prove the inclusion ``$\supset$''. Since $X$ is affine, there exists a nonzero $B$-semi-invariant function $f \in \Bbbk[X]$ that vanishes on all colors of~$X$. Without loss of generality we may assume that $f = f_\lambda$ for some $\lambda \in \Gamma_X$. If $\alpha \in \Pi \setminus \Pi^p_X$ then $\mathcal D_X(\alpha) \ne \varnothing$ by~(\ref{eqn_type_(p)}), hence the line $\Bbbk f_\lambda$ is $P_\alpha$-unstable and $\langle \alpha^\vee, \lambda \rangle > 0$.
\end{proof}

\subsection{Relations between simple roots and colors}

The results in this subsection are extracted from {\cite[\S\S\,2.7,~3.4]{Lu97}; see also~\cite[\S\,30.10]{Tim}}.

Let $X$ be a spherical $G$-variety.

\begin{proposition}
\label{prop_alternative}
For every $\alpha \in \Pi$, exactly one of the following possibilities is realized:
\begin{enumerate}
\renewcommand{\labelenumi}{\textup($p$\textup)}
\renewcommand{\theenumi}{$p$}

\item \label{pp}
$\mathcal D_X(\alpha) = \varnothing$.

\renewcommand{\labelenumi}{\textup($a$\textup)}
\renewcommand{\theenumi}{$a$}

\item \label{aa}
$\alpha \in \Sigma_X$, $\mathcal D_X(\alpha) = \lbrace D^+, D^- \rbrace$, and $\langle \rho_X(D^+), \lambda \rangle + \langle \rho_X(D^-), \lambda \rangle = \langle \alpha^\vee, \lambda \rangle$ for all $\lambda \in \Lambda_X$.

\renewcommand{\labelenumi}{\textup($a'$\textup)}
\renewcommand{\theenumi}{$a'$}

\item \label{aa'}
$2\alpha \in \Sigma_X$, $\mathcal D_X(\alpha) = \lbrace D \rbrace$, and $\langle \rho_X(D), \lambda \rangle = \langle \frac12 \alpha^\vee, \lambda \rangle$ for all $\lambda \in \Lambda_X$.

\renewcommand{\labelenumi}{\textup($b$\textup)}
\renewcommand{\theenumi}{$b$}

\item \label{bb}
$\QQ \alpha \cap \Sigma_X = \varnothing$, $\mathcal D_X(\alpha) = \lbrace D \rbrace$, and $\langle \rho_X(D), \lambda \rangle = \langle \alpha^\vee, \lambda \rangle$ for all $\lambda \in \Lambda_X$.
\end{enumerate}
\end{proposition}

In what follows, by $\mathcal D^a_X$ (resp. $\mathcal D^{a'}_X$, $\mathcal D^b_X$) we denote the union of the sets $\mathcal D_X(\alpha)$ where $\alpha$ runs over all simple roots of type~(\ref{aa}) (resp.~(\ref{aa'}),~(\ref{bb})).

\begin{proposition}
\label{prop_disjoint_union}
The union $\mathcal D_X = \mathcal D_X^a \cup \mathcal D_X^{a'} \cup \mathcal D_X^b$ is disjoint.
\end{proposition}


\subsection{Classification of spherical homogeneous spaces}
\label{subsec_SHS_classification}

\begin{definition} \label{def_sr_of_G}
An element $\sigma \in \mathfrak X(T)$ is called a \textit{spherical root of~$G$} if $\sigma$ is a non-negative linear combination of simple roots of $G$ with coefficients in $\frac12 \ZZ$ such that the following conditions are satisfied:
\begin{enumerate}
\item
if $\sigma \in \ZZ\Delta$ then $\sigma$ appears in Table~\ref{table_spherical_roots};

\item
if $\sigma \notin \ZZ\Delta$ then $2\sigma$ appears in Table~\ref{table_spherical_roots} and its number is marked by an asterisk.
\end{enumerate}
\end{definition}
We denote the set of all spherical roots of $G$ by~$\Sigma(G)$.

\begin{table}[h]

\caption{Spherical roots} \label{table_spherical_roots}

\begin{tabular}{|c|c|c|c|c|}
\hline

No. & \begin{tabular}{c} Type of \\ $\Supp \sigma$\end{tabular} &
$\sigma$ & $\Pi_\sigma$ & Note\\

\hline

$1$ & $\mathsf A_1$ & $\alpha_1$ & $\varnothing$ & \\

\hline

$2$ & $\mathsf A_1$ & $2\alpha_1$ & $\varnothing$ & \\

\hline

$3\lefteqn{^*}$ & $\mathsf A_1 \times \mathsf A_1$ & $\alpha +
\beta$ & $\varnothing$ & \\

\hline

$4$ & $\mathsf A_r$ & $\alpha_1 + \alpha_2 + \ldots + \alpha_r$ &
\begin{tabular}{c}
$\varnothing$ for $r = 2$; \\
\hline $\alpha_2, \alpha_3, \ldots, \alpha_{r-1}$ \\ for $r \ge 3$
\end{tabular}
& $r \ge 2$\\

\hline

$5\lefteqn{^*}$ & $\mathsf A_3$ & $\alpha_1 + 2\alpha_2 + \alpha_3$
& $\alpha_1, \alpha_3$ & \\

\hline

$6$ & $\mathsf B_r$ & $\alpha_1 + \alpha_2 + \ldots + \alpha_r$ &
\begin{tabular}{c}
$\varnothing$ for $r = 2$; \\
\hline $\alpha_2, \alpha_3, \ldots, \alpha_{r-1}$ \\ for $r \ge 3$
\end{tabular}
& $r \ge 2$\\

\hline

$7$ & $\mathsf B_r$ & $2\alpha_1 + 2\alpha_2 + \ldots +
2\alpha_r$ & $\alpha_2, \alpha_3, \ldots, \alpha_r$ & $r \ge 2$\\

\hline

$8\lefteqn{^*}$ & $\mathsf B_3$ & $\alpha_1 + 2\alpha_2 + 3\alpha_3$
& $\alpha_1, \alpha_2$ & \\

\hline

$9$ & $\mathsf C_r$ & $\alpha_1 + 2\alpha_2 + 2\alpha_3 + \ldots +
2\alpha_{r-1} + \alpha_r$ & $\alpha_3, \alpha_4, \ldots,
\alpha_r$ & $r \ge 3$\\

\hline

$10\lefteqn{^*}$ & $\mathsf D_r$ & $2\alpha_1 + 2\alpha_2 + \ldots +
2\alpha_{r-2} + \alpha_{r-1} + \alpha_r$ & $\alpha_2, \alpha_3,
\ldots, \alpha_r$ & $r \ge 4$ \\

\hline

$11$ & $\mathsf F_4$ & $\alpha_1 + 2\alpha_2 + 3\alpha_3 +
2\alpha_4$ & $\alpha_1, \alpha_2, \alpha_3$ & \\

\hline

$12$ & $\mathsf G_2$ & $\alpha_1 + \alpha_2$ & $\varnothing$ & \\

\hline

$13$ & $\mathsf G_2$ & $2\alpha_1 + \alpha_2$ & $\alpha_2$ & \\

\hline

$14$ & $\mathsf G_2$ & $4\alpha_1 + 2\alpha_2$ & $\alpha_2$ & \\

\hline

\end{tabular}
\end{table}

In Table~\ref{table_spherical_roots}, the notation $\alpha_i$ stands for the $i$th simple root of the set $\Supp \sigma$ whenever the Dynkin diagram of $\Supp \sigma$ is connected. If $\Supp \sigma$ is of type $\mathsf A_1 \times \mathsf A_1$, then $\alpha, \beta$ are the two distinct roots in $\Supp \sigma$.

\begin{remark}
Usually, a spherical root of $G$ is (equivalently) defined as an element $\sigma \in \mathfrak X(T)$ such that there exists a spherical $G$-variety $X$ with $\Lambda_X = \ZZ\sigma$ and $\Sigma_X = \lbrace \sigma \rbrace$. In this paper we adopt Definition~\ref{def_sr_of_G} because it is purely combinatorial and hence practical.
\end{remark}

A pair $(\Pi^p, \sigma)$ with $\Pi^p \subset \Pi$ and $\sigma \in \Sigma(G)$ is said to be \textit{compatible} if
\begin{equation} \label{eqn_Pi_p}
\Pi_\sigma \subset \Pi^p \subset \sigma^\perp
\end{equation}
where the set $\Pi_\sigma \subset \Supp \sigma$ is determined as follows:
\[
\Pi_\sigma =
\begin{cases}
\Supp \sigma \cap \sigma^\perp \setminus \lbrace \alpha_r \rbrace &
\text{if } \sigma = \alpha_1 + \alpha_2 + \ldots + \alpha_r
\text{ with support of type } \mathsf B_r; \\
\Supp \sigma \cap \sigma^\perp \setminus \lbrace \alpha_1 \rbrace
& \text{if } \sigma \text{ has support of type } \mathsf C_r; \\
\Supp \sigma \cap \sigma^\perp & \text{otherwise}.
\end{cases}
\]

For the reader's convenience, in the column ``$\Pi_\sigma$'' of Table~\ref{table_spherical_roots} we listed all roots in the set $\Pi_\sigma$ for every spherical root~$\sigma \in \ZZ \Delta$. If $\sigma \in \Sigma(G) \setminus \ZZ \Delta$, then
$\Pi_\sigma = \Pi_{2\sigma}$.

The following definition is due to Luna; see~\cite[\S\,2]{Lu01}.
Our version of this definition is close to~\cite[Definition~30.21]{Tim}.

\begin{definition} \label{def_HSD}
Suppose that $\Lambda$ is a sublattice in~$\mathfrak X(T)$, $\Pi^p$ is a subset of~$\Pi$, $\Sigma \subset \Sigma(G) \cap \Lambda$ is a set consisting of primitive elements in~$\Lambda$, and $\mathcal D^a$ is a finite set equipped with a map $\rho \colon \mathcal D^a \to \Lambda^*$. For every $\alpha \in \Pi \cap \Sigma$, put $\mathcal D(\alpha) = \lbrace D \in \mathcal D^a \mid \langle \rho(D), \alpha \rangle = 1 \rbrace$.

The quadruple $(\Lambda, \Pi^p, \Sigma, \mathcal D^a)$ is called a \textit{homogeneous spherical datum} if it satisfies the following axioms:
\begin{enumerate}
\renewcommand{\labelenumi}{(A\arabic{enumi})}
\renewcommand{\theenumi}{A\arabic{enumi}}

\item \label{A1}
$\langle \rho(D), \sigma \rangle \le 1$ for all $D \in \mathcal D^a$ and $\sigma \in \Sigma$, and the equality is attained if and only if $\sigma = \alpha \in \Pi \cap \Sigma$ and $D \in \mathcal D(\alpha)$;

\item \label{A2}
for every $\alpha \in \Pi \cap \Sigma$, the set $\mathcal D(\alpha)$ contains exactly two elements $D_\alpha^+$ and $D_\alpha^-$, which satisfy $\langle \rho(D_\alpha^+), \lambda \rangle + \langle \rho(D_\alpha^-), \lambda \rangle = \langle \alpha^\vee, \lambda \rangle$ for all $\lambda \in \Lambda$;

\item \label{A3}
the set $\mathcal D^a$ is the union of the sets $\mathcal D(\alpha)$ over all $\alpha \in \Pi \cap \Sigma$;

\setcounter{enumi}{0}
\renewcommand{\labelenumi}{($\Sigma$\arabic{enumi})}
\renewcommand{\theenumi}{$\Sigma$\arabic{enumi}}

\item \label{Sigma1}
if $\alpha \in \Pi \cap \frac12 \Sigma$ then $\langle \alpha^\vee, \lambda \rangle \in 2\ZZ$ for all $\lambda \in \Lambda$;

\item \label{Sigma2}
if $\alpha, \beta \in \Pi$, $\alpha \perp \beta$, and $\alpha + \beta \in \Sigma \cup 2\Sigma$, then $\langle \alpha^\vee, \lambda \rangle = \langle \beta^\vee, \lambda \rangle$ for all $\lambda \in \Lambda$;

\renewcommand{\labelenumi}{(S)}
\renewcommand{\theenumi}{S}

\item
$\Pi^p \subset \Lambda^\perp$ and for every $\sigma \in \Sigma$ the pair $(\Pi^p, \sigma)$ is compatible.
\end{enumerate}
\end{definition}

\begin{theorem}[\cite{Lu01, BraP14, Cu}] \label{thm_bijection_SHS}
The map $O \mapsto (\Lambda_O, \Pi^p_O, \Sigma_O, \mathcal D^a_O)$ is a bijection between \textup($G$-isomorphism classes of\textup) spherical homogeneous spaces of~$G$ and homogeneous spherical data for~$G$.
\end{theorem}

According to this theorem, the quadruple $(\Lambda_O, \Pi^p_O, \Sigma_O, \mathcal D^a_O)$ is said to be the \textit{homogeneous spherical datum} of~$O$, we shall denote it by $\mathscr H_O$.

\subsection{Affine embeddings of spherical homogeneous spaces}

Let $O$ be a spherical homogeneous space of~$G$. A spherical $G$-variety $X$ containing $O$ as an open $G$-orbit is said to be a \textit{$G$-equivariant embedding} (or simply an \textit{embedding}) of~$O$.

\begin{definition}
An embedding $X$ of $O$ is said to be \textit{simple} if $X$
contains exactly one closed $G$-orbit.
\end{definition}

Simple embeddings are classified by strictly convex colored cones.

\begin{definition}[{see \cite[\S\,3]{Kn91}}]
A \textit{colored cone} is a pair $(\mathcal C, \mathcal F)$ with $\mathcal C \subset \mathcal Q_O$ and $\mathcal F \subset \mathcal D_O$ having the following properties:
\begin{enumerate}[label=\textup{(CC\arabic*)},ref=\textup{CC\arabic*}]
\item
$\mathcal C$ is a cone generated by $\rho_O(\mathcal F)$ and finitely many elements of~$\mathcal V_O$;

\item
$\mathcal C^\circ \cap \mathcal V_O \ne \varnothing$.
\end{enumerate}

\noindent A colored cone is said to be \textit{strictly convex} if the following property holds:
\begin{enumerate}[label=\textup{(SCC)},ref=\textup{SCC}]
\item
$\mathcal C$ is strictly convex and $0 \notin \rho_O(\mathcal F)$.
\end{enumerate}
\end{definition}

Let $X$ be a simple embedding of $O$ and let $Y$ be the closed $G$-orbit of $X$. We put $\mathcal F_X = \lbrace D \in \mathcal D_X \mid Y \subset D \rbrace$ and let $\mathcal C_X$ denote the cone in $\mathcal Q_X$ generated by the set $\rho_X(\mathcal B_X \cup \mathcal F_X)$.

\begin{proposition}[{\cite[Theorem~3.1]{Kn91}}]
\label{prop_simple_embeddings} The map $X \mapsto (\mathcal C_X, \mathcal F_X)$ is a bijection between $G$-isomorphism classes of simple embeddings of\, $O$ and strictly convex colored cones in~$\mathcal Q_O$.
\end{proposition}

The following theorem provides a description of all affine
embeddings of~$O$.

\begin{theorem}[{\cite[Theorem~6.7]{Kn91}}]
\label{thm_affinity_criterion} Let $X$ be an embedding of~$O$.
\begin{enumerate}[label=\textup{(\alph*)},ref=\textup{\alph*}]

\item
If $X$ is affine then $X$ is simple.

\item
Suppose that $X$ is simple and let $(\mathcal C, \mathcal F)$ be the corresponding colored cone. Then $X$ is affine if and only if there is an element $\chi \in \Lambda_X$ such that:

\begin{enumerate}[label=\textup{(AE\arabic*)},ref=\textup{AE\arabic*}]

\item \label{AE1}
$\langle v, \chi \rangle \le 0$ for all $v \in \mathcal V_O$;

\item
$\langle q, \chi \rangle = 0$ for all $q \in \mathcal C$;
\label{AE2}

\item
$\langle \rho_O(D), \chi \rangle > 0$ for all $D \in \mathcal D_O \setminus \mathcal F$. \label{AE3}
\end{enumerate}
\end{enumerate}
\end{theorem}

Here is a useful application of the above theorem.

\begin{proposition}[{compare with~\cite[Corollary~15.5]{Tim}}] \label{prop_AC_appl}
Let $\mathcal K \subset \mathcal Q$ be a strictly convex cone generated by $\rho_O(\mathcal D_O)$ and finitely many elements of\,~$\mathcal V_O$. Suppose that $0 \notin \rho_O(\mathcal D_O)$. Then there exists an affine embedding~$X$ of~$O$ such that $\Gamma_X = \Lambda_O \cap \mathcal K^\vee$, where $\mathcal K^\vee$ is considered as a cone in~$(\Lambda_O)_\QQ$.
\end{proposition}

\begin{proof}
Let $\mathcal C$ be the largest face of $\mathcal K$ such that $\mathcal C^\circ \cap \mathcal V_O \ne \varnothing$ and set
\[
\mathcal F = \lbrace D \in \mathcal D_O \mid \rho_O(D) \in \mathcal
C \rbrace.
\]
Then $(\mathcal C, \mathcal F)$ is a colored cone, and the simple embedding $X$ of $O$ corresponding to $(\mathcal C, \mathcal F)$ has the desired properties.
\end{proof}

\section{Affine spherical \texorpdfstring{$G$}{G}-varieties with a prescribed weight monoid}
\label{sec_ASV_description}

\subsection{Spherical roots compatible with a lattice}

Let $\Lambda \subset \mathfrak X(T)$ be a sublattice.

\begin{definition} \label{def_SR_comp_with_lattice}
A spherical root $\sigma \in \Sigma(G)$ is said to be
\textit{compatible with $\Lambda$} if the following properties hold:

\begin{enumerate}[label=\textup{(CL\arabic*)},ref=\textup{CL\arabic*}]
\item \label{CL1}
$\sigma \in \Lambda$ and $\sigma$ is a primitive element of~$\Lambda$;

\item \label{CL2}
the pair $(\Lambda^\perp, \sigma)$ is compatible;

\item \label{CL3}
if $\sigma = \alpha + \beta$ or $\sigma = \frac12(\alpha + \beta)$ for some $\alpha, \beta \in \Pi$ with $\alpha \perp \beta$, then $\langle \alpha^\vee, \lambda \rangle = \langle \beta^\vee, \lambda \rangle$ for all $\lambda \in \Lambda$;

\item \label{CL4}
if $\sigma = 2\alpha$ for some $\alpha \in \Pi$, then $\langle \alpha^\vee, \lambda \rangle \in 2\ZZ$ for all $\lambda \in \Lambda$.
\end{enumerate}
\end{definition}

A geometrical interpretation of this definition is given by

\begin{proposition} \label{prop_comp_with_lattice}
For a spherical root $\sigma \in \Sigma(G)$, the following
conditions are equivalent.

\begin{enumerate}[label=\textup{(\arabic*)},ref=\textup{\arabic*}]
\item \label{comp_with_Lambda1}
$\sigma$ is compatible with $\Lambda$.

\item \label{comp_with_Lambda2}
There exists a spherical homogeneous space $G / H$ with $\Lambda_{G / H} = \Lambda$ and $\Sigma_{G / H} = \lbrace \sigma \rbrace$.
\end{enumerate}
\end{proposition}
\begin{proof}
(\ref{comp_with_Lambda1})$\Rightarrow$(\ref{comp_with_Lambda2})
According to Theorem~\ref{thm_bijection_SHS}, it suffices to find a set $\mathcal D^a$ equipped with a map $\rho \colon \mathcal D^a \to \Lambda^*$ such that $\mathscr H = (\Lambda, \Lambda^\perp, \lbrace \sigma \rbrace, \mathcal D^a)$ is a homogeneous spherical datum. If $\sigma \notin \Pi$, then we take $\mathcal D^a = \varnothing$. In case $\sigma = \alpha \in \Pi$, we take $\mathcal D^a$ to be a set consisting of two elements $D^+$ and $D^-$ such that $\rho(D^+)$ is any element in $\Lambda^*$ with $\langle \rho(D^+), \alpha \rangle = 1$ and $\langle \rho(D^-), \lambda \rangle = \langle \alpha^\vee, \lambda \rangle - \langle \rho(D^+), \lambda \rangle$ for all $\lambda \in \Lambda$. In both cases, one easily checks that $\mathscr H$ is a homogeneous spherical datum.

(\ref{comp_with_Lambda2})$\Rightarrow$(\ref{comp_with_Lambda1})
Thanks to Theorem~\ref{thm_bijection_SHS}, this follows by comparing Definitions~\ref{def_HSD} and~\ref{def_SR_comp_with_lattice}.
\end{proof}


\subsection{Spherical roots compatible with a monoid}
\label{subsec_comp_with_monoid}

Let $\Gamma \subset \Lambda^+$ be a finitely generated and saturated monoid. Set $\mathcal L = (\ZZ \Gamma)^*$, $\mathcal Q = \mathcal L_\QQ = (\QQ \Gamma)^*$ and let $\iota \colon \mathfrak X(T)^* \to \mathcal L$ be the restriction map. Further, let $\mathcal K \subset \mathcal Q$ be the cone dual to $\QQ^+ \Gamma$. Clearly, $\mathcal K$ is strictly convex. Let $\mathcal K^1$ be the set of primitive elements $\varrho$ in $\mathcal L$ such that $\QQ^+ \varrho$ is an extremal ray of~$\mathcal K$. Finally, for every $\sigma \in \ZZ \Gamma$ we put $\mathcal K^1(\sigma) = \lbrace \varrho \in \mathcal K^1 \mid \langle \varrho, \sigma \rangle > 0 \rbrace$.

\begin{definition} \label{def_SR_comp_with_monoid}
A spherical root $\sigma \in \Sigma(G)$ is said to be \textit{compatible with $\Gamma$} if $\sigma$ is compatible with the lattice $\ZZ \Gamma$ and satisfies the following conditions:

\begin{enumerate}[label=\textup{(CM\arabic*)},ref=\textup{CM\arabic*}]
\item \label{CM1}
if $\sigma \notin \Pi$ then for every $\varrho \in \mathcal K^1(\sigma)$ there exists $\delta \in \Pi \setminus \Gamma^\perp$ such that $\iota(\delta^\vee)$ is a positive multiple of~$\varrho$.

\item \label{CM2}
if $\sigma = \alpha \in \Pi$ then there exist two elements $\varrho_1, \varrho_2 \in \mathcal K \cap \mathcal L$ with the following properties:
\begin{enumerate}[label=\textup{(\alph*)},ref=\textup{\alph*}]
\item
$\langle \varrho_1, \alpha \rangle = \langle \varrho_2, \alpha \rangle = 1$;

\item
$\iota(\alpha^\vee) = \varrho_1 + \varrho_2$;

\item
$\mathcal K^1(\alpha) \subset \lbrace \varrho_1, \varrho_2 \rbrace$.
\end{enumerate}
\end{enumerate}
The set of all spherical roots $\sigma \in \Sigma(G)$ compatible with~$\Gamma$ will be denoted by $\Sigma(\Gamma)$.
\end{definition}

\begin{remark} \label{rem_CM2}
It follows from condition (\ref{CM2}) that, for every $\alpha \in \Sigma(\Gamma) \cap \Pi$, at least one of the two elements $\varrho_1, \varrho_2$ lies on an extremal ray of the cone~$\mathcal K$. The latter implies that the two elements $\varrho_1, \varrho_2$ are uniquely determined, up to a permutation.
\end{remark}

With every $\alpha \in \Sigma(\Gamma) \cap \Pi$ we associate a two-element set $\mathcal D(\alpha) = \lbrace D_\alpha^+, D_\alpha^- \rbrace$ equipped with the map $\rho \colon \mathcal D(\alpha) \to \mathcal L$ given by $\rho(D_\alpha^+) = \varrho_1$ and $\rho(D_\alpha^-) = \varrho_2$.

The following proposition provides a geometrical interpretation of spherical roots compatible with~$\Gamma$.

\begin{proposition} \label{prop_comp_with_monoid}
For a spherical root $\sigma \in \Sigma(G)$, the following
conditions are equivalent.

\begin{enumerate}[label=\textup{(\arabic*)},ref=\textup{\arabic*}]
\item \label{comp_with_Gamma1}
$\sigma \in \Sigma(\Gamma)$.

\item \label{comp_with_Gamma2}
There exists an affine spherical $G$-variety $X$ with $\Gamma_X = \Gamma$ and $\Sigma_X = \lbrace \sigma \rbrace$.
\end{enumerate}
\end{proposition}

\begin{proof}
For every spherical root $\sigma \in \Sigma(G) \cap \ZZ \Gamma$, we put $\mathcal V_\sigma = \lbrace q \in \mathcal Q \mid \langle q, \sigma \rangle \le 0 \rbrace$.

(\ref{comp_with_Gamma1})$\Rightarrow$(\ref{comp_with_Gamma2}) We consider two cases.

\textit{Case}~1: $\sigma \notin \Pi$. Then $\mathscr H = (\ZZ \Gamma, \Gamma^\perp, \lbrace \sigma \rbrace, \varnothing)$ is a homogeneous spherical datum.  By Theorem~\ref{thm_bijection_SHS}, there is a spherical homogeneous space $O$ of $G$ such that $\mathscr H_O = \mathscr H$. In view of Propositions~\ref{prop_alternative} and~\ref{prop_disjoint_union}, we thus have
\begin{equation} \label{eqn_colors_case2}
\rho_O(\mathcal D_O) =
\begin{cases}
\lbrace \iota(\gamma^\vee) \mid \gamma \in \Pi \setminus \Gamma^\perp \rbrace & \text{if} \ \sigma \notin 2\Pi;\\
\lbrace \iota(\alpha^\vee)/2 \rbrace \cup \lbrace \iota(\gamma^\vee) \mid \gamma \in \Pi \setminus (\Gamma^\perp \cup \lbrace \alpha \rbrace) \rbrace & \text{if} \ \sigma = 2\alpha \in 2\Pi.
\end{cases}
\end{equation}
In particular, $0 \notin \rho_O(\mathcal D_O)$. It follows from~(\ref{eqn_colors_case2}) and condition~(\ref{CM1}) that the cone $\mathcal K$ is generated by the set $\rho_O(\mathcal D_O)$ and finitely many elements of $\mathcal V_\sigma$. As $\mathcal K$ is strictly convex, by Proposition~\ref{prop_AC_appl} there exists an affine embedding $X$ of~$O$ such that $\Gamma_X = \Gamma$.

\textit{Case}~2: $\sigma = \alpha \in \Pi$. We consider the quadruple $\mathscr H = (\ZZ \Gamma, \Gamma^\perp, \lbrace \alpha \rbrace, \mathcal D(\alpha))$, where $\mathcal D(\alpha)$ is equipped with the above map~$\rho$. It is easily verified that $\mathscr H$ is a homogeneous spherical datum. By Theorem~\ref{thm_bijection_SHS}, there is a spherical homogeneous space $O$ of $G$ such that $\mathscr H_O = \mathscr H$. Then by Propositions~\ref{prop_alternative} and~\ref{prop_disjoint_union} we have
\begin{equation} \label{eqn_colors_case1}
\rho_O(\mathcal D_O) = \lbrace \varrho_1, \varrho_2 \rbrace \cup \lbrace \iota(\gamma^\vee) \mid \gamma \in \Pi \setminus (\Gamma^\perp \cup \lbrace \alpha \rbrace) \rbrace.
\end{equation}
In particular, $0 \notin \rho_O(\mathcal D_O)$. It follows from~(\ref{eqn_colors_case1}) and condition~(\ref{CM2}) that the cone $\mathcal K$ is generated by the set $\rho_O(\mathcal D_O)$ and finitely many elements of $\mathcal V_\sigma$. As $\mathcal K$ is strictly convex, by Proposition~\ref{prop_AC_appl} there exists an affine embedding $X$ of~$O$ such that $\Gamma_X = \Gamma$.

(\ref{comp_with_Gamma2})$\Rightarrow$(\ref{comp_with_Gamma1}) Let $X$ be an affine spherical $G$-variety with $\Gamma_X = \Gamma$ and $\Sigma_X = \lbrace \sigma \rbrace$ and let $O$ be the open $G$-orbit in~$X$. By Proposition~\ref{prop_weight_monoid}(\ref{prop_weight_monoid_b}), we have $\Lambda_X = \ZZ \Gamma_X$. In view of Remark~\ref{rem_invariants}, Proposition~\ref{prop_comp_with_lattice} implies that $\sigma$ is compatible with~$\Lambda_X$. Thanks to Proposition~\ref{prop_weight_monoid}(\ref{prop_weight_monoid_a}), the cone $\mathcal K_X = \mathcal K$ is generated by the set $\rho_X(\mathcal D_X)$ and finitely many elements of~$\mathcal V_X = \mathcal V_\sigma$. Further, Proposition~\ref{prop_weight_monoid}(\ref{prop_weight_monoid_c}) yields $\Pi^p_X = \Gamma^\perp$. Conditions~(\ref{CM1}) and~(\ref{CM2}) now follow from Propositions~\ref{prop_alternative},~\ref{prop_disjoint_union}, and axiom~(\ref{A1}).
\end{proof}

\begin{corollary} \label{crl_SR_are_in_Sigma_Gamma}
Suppose that $X$ is an affine spherical $G$-variety with $\Gamma_X = \Gamma$. Then $\Sigma_X \subset \Sigma(\Gamma)$.
\end{corollary}

\begin{proof}
Thanks to Proposition~\ref{prop_MF_subset}, for every $\sigma \in \Sigma_X$ there exists an affine spherical $G$-variety $Y$ with $\Gamma_Y =\nobreak \Gamma$ and $\Sigma_Y = \lbrace \sigma \rbrace$, hence $\sigma \in \Sigma(\Gamma)$ by Proposition~\ref{prop_comp_with_monoid}.
\end{proof}


\subsection{Admissible sets of spherical roots for a given monoid}
\label{subsec_admissible_sets}

In this subsection, we obtain one of the main results of this paper: a combinatorial description of the affine spherical $G$-varieties with prescribed weight monoid (Theorem~\ref{Theorem-AS}).

We retain all the notation introduced in~\S\,\ref{subsec_comp_with_monoid}.

\begin{definition}\label{def_AS}
A subset $\Sigma \subset \Sigma(\Gamma)$ is said to be \textit{admissible} if it satisfies the following condition:
\begin{enumerate}[label=\textup{(AP)},ref=\textup{AP}]
\item \label{AP}
for every $\alpha \in \Sigma \cap \Pi$, $D \in \mathcal D(\alpha)$, and $\sigma \in \Sigma \setminus \lbrace \alpha \rbrace$, the inequality $\langle \rho(D), \sigma \rangle \le 1$ holds, and the equality is attained if and only if $\sigma = \beta \in \Pi$ and there is $D' \in \mathcal D(\beta)$ with $\rho(D') = \rho(D)$.
\end{enumerate}
\end{definition}

\begin{remark} \label{rem_AP_consequences}
The following statements follow directly from the definition.
\begin{enumerate}[label=\textup{(\alph*)},ref=\textup{\alph*}]
\item \label{rem_AP_consequences_a}
Every $1$-element subset of $\Sigma(\Gamma)$ is admissible.

\item \label{rem_AP_consequences_b}
Every subset $\Sigma \subset \Sigma(\Gamma)$ with $\Sigma \cap \Pi = \varnothing$ is admissible. In particular, $\Sigma(\Gamma)$ is admissible whenever $\Sigma(\Gamma) \cap \Pi = \varnothing$.

\item \label{rem_AP_consequences_c}
A subset $\Sigma \subset \Sigma(\Gamma)$ is admissible if and only if so is every $2$-element subset of~$\Sigma$.

\item \label{rem_AP_consequences_d}
A subset $\lbrace \alpha, \sigma \rbrace \subset \Sigma(\Gamma)$ with $\alpha \in \Pi$ and $\sigma \notin \Pi$ is admissible if and only if $\langle \varrho, \sigma \rangle \le 0$ for every $\varrho \in \rho(\mathcal D(\alpha))$.

\item \label{rem_AP_consequences_e}
If $\Sigma \subset \Sigma(\Gamma)$ is an admissible subset then every subset $\Sigma' \subset \Sigma$ is also admissible.
\end{enumerate}
\end{remark}

\begin{theorem} \label{Theorem-AS}
For a subset $\Sigma \subset \Sigma(\Gamma)$, the following conditions are equivalent.
\begin{enumerate}[label=\textup{(\arabic*)},ref=\textup{\arabic*}]
\item \label{every_pair_is_admissible1}
$\Sigma$ is admissible.

\item \label{every_pair_is_admissible2}
There exists an affine spherical $G$-variety $X$ with $\Gamma_X = \Gamma$ and $\Sigma_X = \Sigma$.
\end{enumerate}
\end{theorem}

\begin{proof}
(\ref{every_pair_is_admissible1})$\Rightarrow$%
(\ref{every_pair_is_admissible2}) First consider the disjoint union $\mathcal S = \bigsqcup \limits_{\alpha \in \Sigma \cap \Pi} \mathcal D(\alpha)$. We introduce an equivalence relation on $\mathcal S$ as follows. For $\alpha, \alpha' \in \Pi \cap \Sigma$, $D \in \mathcal D(\alpha)$, and $D' \in \mathcal D(\alpha')$ we write $D \sim D'$ if and only if one of the following two conditions holds:
\begin{itemize}
\item
$\alpha = \alpha'$ and $D = D'$;

\item
$\alpha \ne \alpha'$ and $\rho(D) = \rho(D')$.
\end{itemize}
Now consider the quotient set $\mathcal D^a = \mathcal S / \!\sim$. By construction, $\mathcal D^a$ is equipped with a well-defined map $\rho \colon \mathcal D^a \to \mathcal L$. For every $\alpha \in \Sigma \cap \Pi$, we shall identify the set $\mathcal D(\alpha)$ with its image in~$\mathcal D^a$. One easily checks that the quadruple $\mathscr H = (\ZZ \Gamma, \Gamma^\perp, \Sigma, \mathcal D^a)$ is a homogeneous spherical datum. By Theorem~\ref{thm_bijection_SHS}, there is a spherical homogeneous space $O$ of $G$ such that $\mathscr H_O = \mathscr H$.  Then Propositions~\ref{prop_alternative} and~\ref{prop_disjoint_union} yield
\begin{equation} \label{eqn_colors}
\rho_O(\mathcal D_O) = \rho(\mathcal D^a) \cup \lbrace \frac12 \iota(\beta^\vee) \mid \beta \in \Pi \cap \frac12 \Sigma \rbrace \cup \lbrace \iota(\beta^\vee) \mid \beta \in \Pi \setminus (\Gamma^\perp \cup \Sigma \cup \frac12 \Sigma \rbrace.
\end{equation}
Note that $0 \notin \rho_O(\mathcal D_O)$.

We now check that the cone $\mathcal K$ is generated by the set $\rho_O(\mathcal D_O)$ and finitely many elements of~$\mathcal V_O$. As $\rho(\mathcal D^a) \subset \mathcal K$ by~(\ref{CM2}), formula~(\ref{eqn_colors}) implies $\rho_O(\mathcal D_O) \subset \mathcal K$. Consequently, it suffices to take an arbitrary element $\varrho \in \mathcal K^1 \setminus \mathcal V_O$ and show that a suitable positive multiple of $\varrho$ lies in $\rho_O(\mathcal D_O)$. Since $\mathcal V_O = \bigcap\limits_{\sigma \in \Sigma} \mathcal V_{\sigma}$, there is a spherical root $\sigma \in \Sigma$ such that $\varrho \in \mathcal K^1(\sigma)$. If $\sigma \in \Pi$ then $\varrho \in \rho_O(\mathcal D_O)$ by~(\ref{CM2}) and~(\ref{eqn_colors}). If $\sigma \notin \Pi$ then by~(\ref{CM1}) there exists $\delta \in \Pi \setminus \Gamma^\perp$ such that $\iota(\delta^\vee)$ is a positive multiple of~$\varrho$. It follows from~(\ref{eqn_colors}) that $\iota(\delta^\vee)$ or $\iota(\delta^\vee)/2$ lies in $\rho_O(\mathcal D_O)$ unless $\delta \in \Pi \cap \Sigma$. But the latter implies $\langle \delta^\vee, \sigma \rangle > 0$, which is impossible because $\delta$ and $\sigma$ are two simple roots in a root system; see \S\,\ref{subsec_comb_inv}.

Thus, the strictly convex cone $\mathcal K$ satisfies all the conditions of Proposition~\ref{prop_AC_appl}, and so there exists an affine embedding $X$ of~$O$ such that $\Gamma_X = \Gamma$.

(\ref{every_pair_is_admissible2})$\Rightarrow$%
(\ref{every_pair_is_admissible1})
Let $X$ be an affine spherical $G$-variety with $\Gamma_X = \Gamma$ and $\Sigma_X = \Sigma$. By Proposition~\ref{prop_weight_monoid}(\ref{prop_weight_monoid_a}), the cone $\mathcal K$ is generated by the set $\rho_X(\mathcal B_X \cup \mathcal D_X)$. Now take any $\alpha \in \Sigma \cap \Pi$. In view of condition~(\ref{CM2}) and~Remark~\ref{rem_CM2}, the set $\mathcal K^1(\alpha)$ is non-empty and is contained in~$\mathcal D(\alpha)$. On the other hand, Propositions~\ref{prop_alternative}, \ref{prop_disjoint_union}, Remark~\ref{rem_invariants}, Theorem~\ref{thm_bijection_SHS}, and axiom~(\ref{A1}) imply $\mathcal K^1(\alpha) \subset \rho_X(\mathcal D_X(\alpha))$, hence $\rho(\mathcal D(\alpha)) = \rho_X(\mathcal D_X(\alpha))$. The latter yields~(\ref{AP}) thanks to axiom~(\ref{A1}).
\end{proof}

\section{Applications to moduli schemes \texorpdfstring{$\mathrm M_\Gamma$}{M_Gamma}}
\label{sec_applications}

Throughout this section, $\Gamma$ stands for a finitely generated and saturated monoid. We retain all the notation introduced at the beginning of~\S\,\ref{subsec_comp_with_monoid}.

\subsection{A combinatorial description of the irreducible components of \texorpdfstring{$\mathrm M_\Gamma$}{M_Gamma}}
\label{subsec_irr_comp}

In this subsection, we apply the results of~\S\,\ref{subsec_admissible_sets} to describe the irreducible components of the moduli scheme~$\mathrm M_\Gamma$.

According to Theorem~\ref{Theorem-AS}, for every admissible subset $\Sigma \subset \Sigma(\Gamma)$ let $X(\Sigma)$ be the affine spherical $G$-variety such that $\Gamma_{X(\Sigma)} = \Gamma$ and $\Sigma_{X(\Sigma)} = \Sigma$.

\begin{theorem} \label{indexationofIrrComp}
The map $\Sigma \mapsto \overline{T_\ad X(\Sigma)}$ is a bijection between the maximal with respect to inclusion admissible subsets of $\Sigma(\Gamma)$ and the irreducible components of~$\mathrm M_\Gamma$. Moreover, $\dim \overline{T_\ad X(\Sigma)} = |\Sigma|$.
\end{theorem}

\begin{proof}
This follows readily from Theorem~\ref{Theorem-AS}, Corollaries~\ref{crl_irr_comp} and~\ref{crl_dimension_of_CX}, and Proposition~\ref{prop_partial_order}.
\end{proof}

In view of Remark~\ref{rem_AP_consequences}(\ref{rem_AP_consequences_a}), the set $\Sigma(\Gamma)$ contains a unique maximal admissible subset if and only if $\Sigma(\Gamma)$ is admissible itself. This along with Theorem~\ref{indexationofIrrComp} yields the following irreducibility criterion for~$\mathrm M_\Gamma$.

\begin{corollary} \label{admissibility-versus-irreducibility}
The following conditions are equivalent.
\begin{enumerate}[label=\textup{(\arabic*)},ref=\textup{\arabic*}]
\item \label{H-SS}
The set $\Sigma(\Gamma)$ is admissible.

\item \label{IRR-mod}
$\mathrm M_\Gamma$ is irreducible.
\end{enumerate}
\end{corollary}

\subsection{The tangent space of \texorpdfstring{$\mathrm M_\Gamma$}{M_Gamma} at~\texorpdfstring{$X_0$}{X_0}}

In this subsection, we present (in a reformulated form) the combinatorial description of the $T_\ad$-module structure in $T_{X_0} \mathrm M_\Gamma$ obtained in~\cite{ACF15}; see Theorem~\ref{thm_tgtspace_description}.
Our version of this description, which will be needed in the remaining part of this section, requires the notions of a $\Gamma$-deviant simple root and a $\Gamma$-loose spherical root.

\begin{definition} \label{def_DR}
A root $\alpha \in \Pi$ is said to be \textit{$\Gamma$-deviant} if $\alpha \in \ZZ\Gamma$ and there exist two distinct elements $\varrho_1, \varrho_2 \in \mathcal K^1$ with the following properties:
\begin{enumerate}[label=\textup{(DR\arabic*)},ref=\textup{DR\arabic*}]
\item \label{DR1}
$\langle \varrho_1, \alpha \rangle = \langle \varrho_2, \alpha \rangle = 1$;

\item \label{DR2}
$\iota(\alpha^\vee) \in (\QQ^+\varrho_1 + \QQ^+\varrho_2) \setminus \lbrace 2\varrho_1, \varrho_1 + \varrho_2, 2\varrho_2 \rbrace$;

\item \label{DR3}
$\mathcal K^1(\alpha) = \lbrace \varrho_1, \varrho_2 \rbrace$.
\end{enumerate}
\end{definition}

The set of all $\Gamma$-deviant roots will be denoted by $\Dev(\Gamma)$.

\begin{remark} \label{rem_deviant_roots}
It follows directly from the definition that every $\alpha \in \Dev(\Gamma)$ has the following properties:
\begin{enumerate}[label=\textup{(\alph*)},ref=\textup{\alph*}]
\item \label{rem_deviant_roots_a}
$\alpha$ is primitive in the lattice $\ZZ\Gamma$;

\item
$\alpha$ is compatible with the lattice $\ZZ \Gamma$;

\item \label{rem_deviant_roots_c}
$\alpha \notin \Sigma(\Gamma)$.
\end{enumerate}
\end{remark}

The following proposition shows that the set $\Dev(\Gamma)$ is empty for a wide class of monoids~$\Gamma$.

\begin{proposition} \label{prop_deviant_empty}
Suppose that $\Gamma = \Gamma_0 \oplus \Lambda_0$ where $\Gamma_0$ is a free monoid and $\Lambda_0$ is a lattice with $\Lambda_0^\perp = \Pi$ \textup(that is, $\Lambda_0 \subset \mathfrak X(C)$\textup). Then $\Dev(\Gamma) = \varnothing$. In particular, $\Dev(\Gamma) = \varnothing$ whenever $\Gamma$ is free.
\end{proposition}

\begin{proof}
Suppose that $\alpha \in \ZZ\Gamma \cap \Pi$ and two distinct elements $\varrho_1, \varrho_2 \in \mathcal K^1$ satisfy conditions~(\ref{DR1})--(\ref{DR3}). Then $\lbrace \varrho_1, \varrho_2 \rbrace$ is a part of a basis of~$\mathcal L$ hence $\iota(\alpha^\vee) = b_1 \varrho_1 + b_2 \varrho_2$ for some $b_1, b_2 \in \ZZ$. In this case, one easily checks that conditions~(\ref{DR1}), (\ref{DR2}) cannot hold simultaneously.
\end{proof}

Examples of $\Gamma$ with $\Dev(\Gamma) \ne \varnothing$ are given in \S\,\ref{subsec_examples_non-reduced}.

For our description of $T_{X_0} \mathrm M_\Gamma$, we shall need the following  lemma.

\begin{lemma} \label{lemma_deviant+compatible}
For an element $\alpha \in \ZZ \Gamma \cap \Pi$, the following conditions are equivalent.
\begin{enumerate}[label=\textup{(\arabic*)},ref=\textup{\arabic*}]
\item \label{lemma_deviant+compatible_1}
$\alpha \in \Dev(\Gamma) \cup \Sigma(\Gamma)$.

\item \label{lemma_deviant+compatible_2}
There exist two elements $\varrho_1, \varrho_2 \in \mathcal K \cap \mathcal L$ satisfying the following conditions:
\begin{enumerate}[label=\textup{(\alph*)},ref=\textup{\alph*}]
\item \label{DCa}
$\langle \varrho_1, \alpha \rangle = \langle \varrho_2, \alpha \rangle = 1$;

\item \label{DCb}
$\iota(\alpha^\vee) \in (\QQ^+\varrho_1 + \QQ^+\varrho_2) \setminus \lbrace 2\varrho_1, 2\varrho_2 \rbrace$;

\item \label{DCc}
$\mathcal K^1(\alpha) \subset \lbrace \varrho_1, \varrho_2 \rbrace$.
\end{enumerate}
\end{enumerate}
\end{lemma}

\begin{proof}
The implication (\ref{lemma_deviant+compatible_1})$\Rightarrow$(\ref{lemma_deviant+compatible_2}) follows directly from Definitions~\ref{def_SR_comp_with_monoid} and~\ref{def_DR}. To prove the converse implication, suppose that two elements $\varrho_1, \varrho_2 \in \mathcal K \cap \mathcal L$ satisfy conditions (\ref{DCa})--(\ref{DCc}). Clearly, $\mathcal K^1(\alpha) \ne \varnothing$, and so we have two cases.

\textit{Case}~1: $|\mathcal K^1(\alpha)|= 2$, that is, $\mathcal K^1(\alpha) = \lbrace \varrho_1, \varrho_2 \rbrace$. It follows from Definitions~\ref{def_SR_comp_with_monoid} and~\ref{def_DR} that $\alpha \in \Dev(\Gamma) \cup \Sigma(\Gamma)$.

\textit{Case}~2: $|\mathcal K^1(\alpha)| = 1$. Without loss of generality we assume that $\mathcal K^1(\alpha) = \lbrace \varrho_1 \rbrace$. We claim that $\alpha \in \Sigma(\Gamma)$. To check condition~(\ref{CM2}) it suffices to prove that the element $\varrho'_2 = \iota(\alpha^\vee) - \varrho_1$ lies in the cone~$\mathcal K$. Let $a,b \in \QQ^+ \setminus \lbrace 0 \rbrace$ be such that $\iota(\alpha^\vee) = a\varrho_1 + b\varrho_2$; note that $a+b = 2$. Next, as $\varrho_2 \in \mathcal K$ there is an expression $\varrho_2 = c\varrho_1 + \tau$ where $c \in \QQ^+$ and $\tau$ is an element of the cone spanned by the set $\mathcal K^1 \setminus \lbrace \varrho_1 \rbrace$. Since $\langle \tau, \alpha \rangle \le 0$, it follows from~(\ref{DCa}) that $c \ge 1$. We have $\varrho'_2 = (a+bc-1) \varrho_1 + b \tau$ with $a + bc - 1 \ge a+b - 1 = 1$, and so $\varrho'_2 \in \mathcal K$.
\end{proof}

\begin{definition}
A spherical root $\sigma \in \Sigma(G)$ is said to be \textit{$\Gamma$-loose\footnote{The term is taken from~\cite[\S\,2.2]{BL} where it was used in a similar situation.}} if $\sigma \in \Sigma(\Gamma)$ and one of the following conditions holds:
\begin{enumerate}[label=\textup{(LR\arabic*)},ref=\textup{LR\arabic*}]
\item
$\sigma \notin \ZZ \Pi$;

\item
$\sigma = \alpha \in \Pi$ and $\rho(\mathcal D(\alpha)) = \lbrace \iota(\alpha^\vee)/2 \rbrace$;

\item
$\sigma = \alpha_1 + \ldots + \alpha_r$ with $\Supp \sigma$ of type~$\mathsf B_r$ \textup($r \ge 2$\textup) and $\alpha_r \in \Gamma^\perp$;

\item
$\sigma = 2\alpha_1 + \alpha_2$ with $\Supp \sigma$ of type~$\mathsf G_2$.
\end{enumerate}
\end{definition}

Note that $2\sigma \in \Sigma(G) \cap \ZZ \Pi$ for every $\Gamma$-loose $\sigma \in \Sigma(G)$.

For every $\sigma \in \Sigma(\Gamma)$ we define the element $\overline \sigma \in \lbrace \sigma, 2\sigma \rbrace$ as follows:
\[
\overline \sigma =
\begin{cases}
2\sigma & \text{if } \sigma \text{ is $\Gamma$-loose}; \\
\sigma & \text{otherwise}.
\end{cases}
\]

The following theorem, which is a particular case of~\cite[Theorem~2]{Lo09a} (see also \cite[Theorem~4.20]{ACF15}), explains the role played by $\Gamma$-loose spherical roots for affine spherical $G$-varieties.

\begin{theorem} \label{thm_overline_Sigma}
Suppose that $X$ is an affine spherical $G$-variety with weight monoid~$\Gamma$. Then $\overline \Sigma_X = \lbrace \overline \sigma \mid \sigma \in \Sigma_X \rbrace$.
\end{theorem}

Next, we define the set
\[
\overline \Sigma(\Gamma) = \lbrace \overline \sigma \mid \sigma \in \Sigma(\Gamma) \rbrace \subset \Sigma(G).
\]
Note that $\sigma \mapsto \overline \sigma$ is a natural bijection between $\Sigma(\Gamma)$ and $\overline \Sigma(\Gamma)$.

\begin{remark} \label{rem_intersection_is_empty}
$\overline \Sigma(\Gamma) \cap \Dev(\Gamma) = \varnothing$.
\end{remark}

We now put
\[
\Phi(\Gamma)= \lbrace \sigma \in \mathfrak X(T_\ad) \mid -\sigma \textit{ is a $T_\ad$-weight of~$T_{X_0} \mathrm M_\Gamma$}\rbrace.
\]
In other words, $\Phi(\Gamma)$ is the set of $T_\ad$-weights in the cotangent space of $\mathrm M_\Gamma$ at~$X_0$.

The following theorem is a reformulation of~\cite[Theorem~3.1]{ACF15}.

\begin{theorem}
\label{thm_tgtspace_description}
The tangent space $T_{X_0}\mathrm M_\Gamma$ is a multiplicity-free $T_\ad$-module\footnote{Here the term ``multiplicity-free $T_\ad$-module'' should not be mixed with ``multiplicity-free $T_\ad$-variety'', which has a different meaning.}. Moreover, $\Phi(\Gamma) = \overline \Sigma(\Gamma) \cup \Dev(\Gamma)$.
\end{theorem}

\begin{proof}
This follows by comparing \cite[Theorem~3.1]{ACF15} with the definitions of $\Sigma(\Gamma)$, $\overline \Sigma(\Gamma)$, $\Dev(\Gamma)$ and taking into account Lemma~\ref{lemma_deviant+compatible}.
\end{proof}

Combining this theorem together with Proposition~\ref{prop_deviant_empty}, we obtain

\begin{corollary}
Suppose that $\Gamma$ is free. Then $\Phi(\Gamma) = \overline \Sigma(\Gamma)$.
\end{corollary}

\subsection{A smoothness criterion for \texorpdfstring{$\mathrm M_\Gamma$}{M_Gamma}}

\begin{theorem} \label{thm_smoothness_criterion}
The following conditions are equivalent.
\begin{enumerate}[label=\textup{(\arabic*)},ref=\textup{\arabic*}]
\item \label{thm_smoothness_criterion_1}
The set $\Sigma(\Gamma)$ is admissible and $\Dev(\Gamma) = \varnothing$.

\item \label{thm_smoothness_criterion_2}
$\mathrm M_\Gamma$ is an affine space \textup(as a scheme\textup).
\end{enumerate}
\end{theorem}

\begin{proof}
(\ref{thm_smoothness_criterion_1})$\Rightarrow$%
(\ref{thm_smoothness_criterion_2})
Theorems~\ref{indexationofIrrComp} and~\ref{thm_tgtspace_description} imply that $\mathrm M_\Gamma$ is smooth at~$X_0$, and so $\mathrm M_\Gamma$ is an affine space by Theorem~\ref{thm_smoothness}.

(\ref{thm_smoothness_criterion_2})$\Rightarrow$%
(\ref{thm_smoothness_criterion_1})
The set $\Sigma(\Gamma)$ is admissible by Corollary~\ref{admissibility-versus-irreducibility}. Next, thanks to Theorem~\ref{indexationofIrrComp}, there exists an affine spherical $G$-variety $X$ with $\Gamma_X = \Gamma$ such that $\mathrm M_\Gamma = \overline{T_\ad X}$. As $\mathrm M_\Gamma$ is an affine space, Theorem~\ref{thm_orbit_closure_of_X} yields $\Phi(\Gamma) = \overline \Sigma_X$. Now Theorem~\ref{thm_tgtspace_description}, Remark~\ref{rem_deviant_roots}(\ref{rem_deviant_roots_a},\,\ref{rem_deviant_roots_c}), and Corollary~\ref{crl_SR_are_in_Sigma_Gamma} imply $\Dev(\Gamma) = \varnothing$.
\end{proof}

\begin{corollary} \label{crl_reducedness_criterion}
Suppose that the set $\Sigma(\Gamma)$ is admissible \textup(or, equivalently, $\mathrm M_\Gamma$ is irreducible\textup). Then the following conditions are equivalent.
\begin{enumerate}[label=\textup{(\arabic*)},ref=\textup{\arabic*}]
\item \label{crl_reducedness_criterion_1}
$\Dev(\Gamma) = \varnothing$.

\item \label{crl_reducedness_criterion_2}
$\mathrm M_\Gamma$ is reduced.
\end{enumerate}
\end{corollary}

\begin{proof}
(\ref{crl_reducedness_criterion_1})$\Rightarrow$%
(\ref{crl_reducedness_criterion_2}) This follows from Theorem~\ref{thm_smoothness_criterion}.

(\ref{crl_reducedness_criterion_2})$\Rightarrow$%
(\ref{crl_reducedness_criterion_1}) Applying Theorems~\ref{indexationofIrrComp}, \ref{thm_orbit_closure_of_X}, \ref{thm_root_monoid_is_free}, and~\ref{thm_overline_Sigma} we find that $\Phi(\Gamma) = \overline \Sigma(\Gamma)$, whence $\Dev(\Gamma) = \varnothing$ by Theorem~\ref{thm_tgtspace_description} and Remark~\ref{rem_intersection_is_empty}.
\end{proof}

\subsection{Sufficient conditions for \texorpdfstring{$\mathrm M_\Gamma$}{M_Gamma} to be irreducible and/or smooth}

To establish such conditions, we shall need the three following lemmas.

\begin{lemma} \label{lemma_alpha_sigma}
Let $\alpha \in \Sigma(\Gamma) \cap \Pi$ and $\sigma \in \Sigma(\Gamma) \setminus \Pi$. Suppose that $\rho(\mathcal D(\alpha)) = \lbrace \iota(\alpha^\vee)/2 \rbrace$. Then the set $\lbrace \alpha, \sigma \rbrace$ is admissible.
\end{lemma}

\begin{proof}
Thanks to Remark~\ref{rem_AP_consequences}(\ref{rem_AP_consequences_d}), we need to show that $\langle \alpha^\vee, \sigma \rangle \le 0$. If $\alpha\not\in \Supp \sigma$ then the latter inequality holds automatically, therefore we may assume that $\alpha \in \Supp \sigma$. Since both pairs $(\Gamma^\perp, \alpha)$ and $(\Gamma^\perp, \sigma)$ are compatible, it follows from~(\ref{eqn_Pi_p}) that $\Pi_\sigma \subset \alpha^\perp$, that is, $\alpha \perp \Pi_\sigma$. An inspection of Table~\ref{table_spherical_roots} shows that the conditions $\langle \alpha^\vee, \sigma \rangle > 0$ and $\alpha \perp \Pi_\sigma$ can hold simultaneously only in one of the following three cases:
\begin{enumerate}[label=\textup{(\arabic*)},ref=\textup{\arabic*}]
\item
$\sigma = \alpha_1 + \alpha_2$ with $\Supp \sigma$ of type $\mathsf A_2$ and $\alpha \in \lbrace \alpha_1, \alpha_2 \rbrace$;

\item
$\sigma = \alpha_1 + \alpha_2$ with $\Supp \sigma$ of type $\mathsf B_2$ and $\alpha = \alpha_1$;

\item
$\sigma = \alpha_1 + \alpha_2$ with $\Supp \sigma$ of type $\mathsf G_2$ and $\alpha = \alpha_2$.
\end{enumerate}
Further, the condition $\iota(\alpha^\vee)/2 \in \mathcal L$ implies $\langle \alpha^\vee, \sigma \rangle \in 2\ZZ$, which is not the case in any of the above three situations. Thus $\langle \alpha^\vee, \sigma \rangle \le 0$.
\end{proof}

\begin{lemma} \label{lemma_not_on_an_extremal_ray}
Suppose that $\alpha \in (\overline \Sigma(\Gamma) \cap \Pi) \cup \Dev(\Gamma)$. Then $\iota(\alpha^\vee)$ does not lie on an extremal ray of the cone~$\mathcal K$.
\end{lemma}

\begin{proof}
This follows by comparing the definitions of the sets $\Sigma(\Gamma)$, $\overline \Sigma(\Gamma)$, and $\Dev(\Gamma)$.
\end{proof}

\begin{lemma} \label{lemma_nonempty}
Suppose that $\alpha \in (\Sigma(\Gamma) \cap \Pi) \cup \Dev(\Gamma)$ and $X$ is an affine spherical $G$-variety with $\Gamma_X = \Gamma$. Then $\mathcal D_X(\alpha) \ne \varnothing$. Moreover,
\[
|\mathcal D_X(\alpha)| =
\begin{cases}
2 & \text{if } \alpha \in \Sigma_X;\\
1 & \text{otherwise.}
\end{cases}
\]
\end{lemma}

\begin{proof}
It follows from the hypothesis that $\alpha \in \ZZ \Gamma$, which implies $\alpha \notin \Pi^p_X$ by Proposition~\ref{prop_weight_monoid}(\ref{prop_weight_monoid_c}). Now the claim follows from Proposition~\ref{prop_alternative}.
\end{proof}

\begin{proposition} \label{prop_irr_cond1}
Suppose that $\rho(\mathcal D(\alpha)) \subset \mathcal K^1$ for every $\alpha \in \Sigma(\Gamma) \cap \Pi$. Then
\begin{enumerate}[label=\textup{(\alph*)},ref=\textup{\alph*}]
\item \label{prop_irr_cond1_a}
the set $\Sigma(\Gamma)$ is admissible;

\item \label{prop_irr_cond1_b}
$\mathrm M_\Gamma$ is irreducible.
\end{enumerate}
\end{proposition}

\begin{proof}
(\ref{prop_irr_cond1_a}) According to Remark~\ref{rem_AP_consequences}(\ref{rem_AP_consequences_b},\,\ref{rem_AP_consequences_c}), it is enough to show that every set $\lbrace \alpha, \sigma \rbrace$ with $\alpha \in \Sigma(\Gamma) \cap \Pi$ and $\sigma \in \Sigma(\Gamma) \setminus \lbrace \alpha \rbrace$ is admissible.

\textit{Case~}1: $\sigma \notin \Pi$. Assume that there is $D \in \mathcal D(\alpha)$ such that $\langle \rho(D), \sigma \rangle > 0$. Since $\rho(D) \in \mathcal K^1$, it follows from (\ref{CM1}) that $\rho(D)$ is proportional to $\iota(\beta^\vee)$ for some $\beta \in \Pi$. As $\langle \rho(D), \alpha \rangle = 1$, we obtain $\beta = \alpha$ and hence $\rho(\mathcal D(\alpha)) = \lbrace \iota(\alpha^\vee)/2\rbrace$. Now the claim is implied by Lemma~\ref{lemma_alpha_sigma}.

\textit{Case~}2: $\sigma = \beta \in \Pi$. Assume that there is $D \in \mathcal D(\alpha)$ such that $\langle \rho(D), \beta \rangle > 0$. Since $\rho(D) \in \mathcal K^1$, it follows from (\ref{CM2}) that $\rho(D) = \rho(D')$ for some $D' \in \mathcal D(\beta)$.

Part~(\ref{prop_irr_cond1_b}) follows from~(\ref{prop_irr_cond1_a}) thanks to Corollary~\ref{admissibility-versus-irreducibility}.
\end{proof}

The next proposition describes a class of monoids $\Gamma$ for which $\mathrm M_\Gamma$ is an affine space.

\begin{proposition} \label{prop_irr_cond_free}
Suppose that $\Gamma = \Gamma_0 \oplus \Lambda_0$ where $\Gamma_0$ is a free monoid with minimal set of generators $\mathrm E$ and $\Lambda_0$ is a lattice with $\Lambda_0^\perp = \Pi$ \textup(that is, $\Lambda_0 \subset \mathfrak X(C)$\textup). Suppose that every $\alpha \in \Pi$ satisfies one of the following conditions:
\begin{enumerate}[label=\textup{(\arabic*)},ref=\textup{\arabic*}]
\item \label{prop_irr_cond_free_1}
$\langle \alpha^\vee, \lambda\rangle > 0$ for at most one $\lambda \in \mathrm E$;

\item \label{prop_irr_cond_free_2}
$\langle \alpha^\vee, \lambda \rangle \le 1$ for all $\lambda \in \mathrm E$.
\end{enumerate}
Then $\mathrm M_\Gamma$ is an affine space.
\end{proposition}

\begin{proof}
Proposition~\ref{prop_deviant_empty} yields $\Dev(\Gamma) = \varnothing$, hence by Theorem~\ref{thm_smoothness_criterion} and Proposition~\ref{prop_irr_cond1} it suffices to show that $\rho(\mathcal D(\alpha)) \subset \mathcal K^1$ for every $\alpha \in \Sigma(\Gamma) \cap \Pi$.

For every $\lambda \in \mathrm E$, let $\varrho_\lambda \in \mathcal K^1$ be the respective dual element.

Take an arbitrary root $\alpha \in \Sigma(\Gamma) \cap \Pi$. Clearly, $\iota(\alpha^\vee) \ne 0$. If $\alpha$ satisfies (\ref{prop_irr_cond_free_1}) then $\iota(\alpha^\vee)$ lies on an extremal ray of $\mathcal K$, which by Lemma~\ref{lemma_not_on_an_extremal_ray} implies $\rho(\mathcal D(\alpha)) = \lbrace \iota(\alpha^\vee)/2 \rbrace \subset \mathcal K^1$. In what follows we assume that $\alpha$ satisfies~(\ref{prop_irr_cond_free_2}). As $\alpha \in \Sigma(\Gamma)$, there is an expression $\alpha = \sum \limits_{\lambda \in \mathrm E} c_\lambda \lambda + \mu$ where $c_\lambda \in \ZZ$ and $\mu \in \Lambda_0$. Note that $c_\lambda = \langle \varrho_\lambda, \alpha \rangle$ for all $\lambda \in \mathrm E$. It follows from (\ref{CM2}) that $c_\lambda \le 1$ for all $\lambda \in \mathrm E$. Now
\[
\langle \alpha^\vee, \alpha \rangle = 2 = \sum \limits_{\lambda \in \mathrm E} c_\lambda \langle \alpha^\vee, \lambda \rangle,
\]
which in view of (\ref{prop_irr_cond_free_2}) implies that there exist two distinct elements $\lambda_1, \lambda_2 \in \mathrm E$ such that $c_{\lambda_1} = 1$ and $c_{\lambda_2} = 1$. It follows from (\ref{CM2}) that $\rho(\mathcal D(\alpha)) = \lbrace \varrho_{\lambda_1}, \varrho_{\lambda_2} \rbrace \subset \mathcal K^1$.
\end{proof}

\begin{remark}
As we shall see in Remark~\ref{rem_factorial_alt} below, the weight monoid of any affine spherical $G$-variety $X$ with $\Bbbk[X]$ a unique factorization domain is of the form described in Proposition~\ref{prop_irr_cond_free}. In particular, so is the weight monoid of any spherical $G$-module.
\end{remark}

\begin{proposition} \label{prop_irr_cond2}
Suppose that there exists an affine spherical $G$-variety $X$ with $\Gamma_X = \Gamma$ such that $\overline \Sigma(\Gamma) \cap \Pi \subset \Sigma_X$ and $\rho_X(D)$ lies on an extremal ray of~$\mathcal K$ for every $D \in \mathcal D_X$. Then $\mathrm M_\Gamma$ is an affine space.
\end{proposition}

\begin{proof}
By Theorem~\ref{thm_smoothness_criterion}, we need to prove that the set $\Sigma(\Gamma)$ is admissible and $\Dev(\Gamma) = \varnothing$.

In view of Proposition~\ref{prop_irr_cond1}(\ref{prop_irr_cond1_a}), to check the admissibility of $\Sigma(\Gamma)$ it suffices to show that $\rho(\mathcal D(\alpha)) \subset \mathcal K^1$ for every $\alpha \in \Sigma(\Gamma) \cap \Pi$.

\textit{Case}~1: $\alpha \in \overline \Sigma(\Gamma)$. The hypotheses imply $\alpha \in \Sigma_X$, whence $\rho(\mathcal D(\alpha)) = \rho(\mathcal D_X(\alpha)) \subset \mathcal K^1$.

\textit{Case}~2: $\alpha \notin \overline \Sigma(\Gamma)$. Then $\rho(\mathcal D(\alpha)) = \lbrace \iota(\alpha^\vee)/2 \rbrace$ by the definition of~$\overline \Sigma(\Gamma)$, and Remark~\ref{rem_CM2} yields $\rho(\mathcal D(\alpha)) \subset \mathcal K^1$.

We now show that $\Dev(\Gamma) = \varnothing$. Take any $\alpha \in \Dev(\Gamma)$. By Lemma~\ref{lemma_nonempty}, the set $\mathcal D_X(\alpha)$ contains a unique element~$D$. Proposition~\ref{prop_alternative} then implies that $\rho(D)$ is proportional to $\iota(\alpha^\vee)$, and so $\iota(\alpha^\vee)$ lies on an extremal ray of~$\mathcal K^1$, which contradicts Lemma~\ref{lemma_not_on_an_extremal_ray}.
\end{proof}

\begin{proposition} \label{prop_factorial_case}
Suppose that there exists an affine spherical $G$-variety $X$ with $\Gamma_X = \Gamma$ such that $\Bbbk[X]$ is a unique factorization domain. Then $\mathrm M_\Gamma$ is an affine space.
\end{proposition}

\begin{proof}
The claim will follow as soon as we check the conditions of Proposition~\ref{prop_irr_cond2}.

It is well known that under our hypotheses all the elements $\rho_X(D)$ with $D \in \mathcal B_X \cup \mathcal D_X$ are linearly independent in~$\mathcal L$; we recall the proof for convenience of the reader. As $\Bbbk[X]$ is a unique factorization domain, the divisor class group of $X$ is trivial. Consequently, for every $D \in \mathcal B_X \cup \mathcal D_X$ there exists a function $f_D \in \Bbbk[X]$ such that $D$ is the divisor of zeros of~$f_D$. As $D$ is $B$-stable, $f_D$ is $B$-semi-invariant, and we let $\lambda_D \in \Gamma$ be the weight of~$f_D$. Then for all $D, D' \in \mathcal B_X \cup \mathcal D_X$ one has
\begin{equation} \label{eqn_lin_ind}
\langle \rho_X(D), \lambda_{D'} \rangle =
\begin{cases}
1 & \text{ if } D = D'; \\
0 & \text{ if } D \ne D'.
\end{cases}
\end{equation}
It follows that all the elements $\rho_X(D)$ with $D \in \mathcal B_X \cup \mathcal D_X$ are linearly independent in~$\mathcal L$. Since these elements generate the cone $\mathcal K$ (see Proposition~\ref{prop_weight_monoid}(\ref{prop_weight_monoid_a})), we get $\mathcal \rho_X(\mathcal D_X) \subset \mathcal K^1$.

To complete the proof, it suffices to show that $\overline \Sigma(\Gamma) \cap \Pi \subset \Sigma_X$. Take any $\alpha \in \overline \Sigma(\Gamma) \cap\nobreak \Pi$ and assume that $\alpha \notin \Sigma_X$. Then by Lemma~\ref{lemma_nonempty} the set $\mathcal D_X(\alpha)$ contains a unique element~$D$. It follows from Proposition~\ref{prop_alternative} that $\iota(\alpha^\vee)$ lies on an extremal ray of the cone $\mathcal K$, which contradicts Lemma~\ref{lemma_not_on_an_extremal_ray}.
\end{proof}

\begin{remark} \label{rem_factorial_alt}
Suppose that $X$ is an affine spherical $G$-variety with $\Gamma_X = \Gamma$ such that $\Bbbk[X]$ is a unique factorization domain. Then, combining relations~(\ref{eqn_lin_ind}) with Proposition~\ref{prop_alternative}, it is easy to see that the weight monoid of~$X$ satisfies the conditions of Proposition~\ref{prop_irr_cond_free} with $\mathrm E = \lbrace \lambda_D \mid D \in \mathcal B_X \cup \mathcal D_X \rbrace$ and $\Lambda_0$ the lattice of weights of invertible $G$-semi-invariant regular functions\footnote{In fact, every invertible regular function on~$X$ is automatically $G$-semi-invariant by~\cite[Theorem~1]{Ro61}.} on~$X$. This gives an alternative proof of Proposition~\ref{prop_factorial_case}.
\end{remark}

The following statement recovers the main results of the papers \cite{PvS12, PvS16}.

\begin{corollary} \label{crl_spherical_modules}
Suppose that there exists a spherical $G$-module $V$ with $\Gamma_V = \Gamma$. Then
\begin{enumerate}[label=\textup{(\alph*)},ref=\textup{\alph*}]
\item \label{crl_spherical_modules_a}
$\mathrm M_\Gamma$ is an affine space;

\item \label{crl_spherical_modules_b}
$\mathrm M_\Gamma = \overline{T_\ad V}$.
\end{enumerate}
\end{corollary}

\begin{proof}
As $\Bbbk[V]$ is a unique factorization domain, part (\ref{crl_spherical_modules_a}) follows from Proposition~\ref{prop_factorial_case}. Part~(\ref{crl_spherical_modules_b}) is implied by \cite[Corollary~2.9]{AB} because $V$ is smooth.
\end{proof}

We now describe one more class of monoids $\Gamma$ for which $\mathrm M_\Gamma$ is an affine space.

\begin{definition}
A finitely generated monoid $\Gamma \subset \Lambda^+$ is called \textit{$G$-saturated} if
\[
\Gamma=\mathbb Z\Gamma\cap \Lambda^+.
\]
\end{definition}

\begin{remark}
Every $G$-saturated monoid is automatically saturated.
\end{remark}

\begin{remark} \label{rem_G-saturated}
A monoid $\Gamma$ is $G$-saturated if and only if its dual cone $\mathcal K$ is generated by the set $\lbrace\iota(\gamma^\vee) \mid \gamma \in \Pi \backslash \Pi^p \rbrace$.
\end{remark}

\begin{theorem} \label{thm_G-saturated}
Suppose that $\Gamma$ is $G$-saturated. Then
\begin{enumerate}[label=\textup{(\alph*)},ref=\textup{\alph*}]
\item \label{thm_G-saturated_a}
$\Dev(\Gamma) = \varnothing$;

\item \label{thm_G-saturated_b}
the set $\Sigma(\Gamma)$ is admissible;

\item \label{thm_G-saturated_c}
$\mathrm M_\Gamma$ is an affine space.
\end{enumerate}
\end{theorem}

\begin{proof}
(\ref{thm_G-saturated_a}) In view of Remark~\ref{rem_G-saturated}, for every $\alpha \in \ZZ \Gamma \cap \Pi$ the set $\mathcal K^1(\alpha)$ contains a unique element, which is proportional to~$\iota(\alpha^\vee)$. Hence $\Dev(\Gamma) = \varnothing$ by Lemma~\ref{lemma_not_on_an_extremal_ray}.

(\ref{thm_G-saturated_b})
Take any $\alpha \in \Sigma(\Gamma) \cap \Pi$. Lemma~\ref{lemma_not_on_an_extremal_ray} together with Remark~\ref{rem_G-saturated} imply that $\rho(\mathcal D(\alpha)) = \lbrace \iota(\alpha^\vee)/2 \rbrace$, whence $\rho(\mathcal D(\alpha)) \subset \mathcal K^1$. Then the set $\Sigma(\Gamma)$ is admissible by Proposition~\ref{prop_irr_cond1}(\ref{prop_irr_cond1_a}).

(\ref{thm_G-saturated_c}) This follows from~(\ref{thm_G-saturated_a}) and~(\ref{thm_G-saturated_b}) thanks to Theorem~\ref{thm_smoothness_criterion}.
\end{proof}

\begin{remark}
For the case where $\Gamma$ is free and $G$-saturated, the fact that $\mathrm M_\Gamma$ is an affine space was known before thanks to the papers~\cite{Jan,BraCu08}.
\end{remark}

\subsection{Examples of reducible \texorpdfstring{$\mathrm M_\Gamma$}{M_Gamma}}
\label{subsec_examples_reducible}

Recall from Theorem~\ref{indexationofIrrComp} that the irreducible components of $\mathrm M_\Gamma$ are in bijection with the maximal admissible subsets of $\Sigma(\Gamma)$.

In all the five examples presented in this subsection, the monoid $\Gamma$ is free, $\Sigma(\Gamma) = \lbrace \sigma_1, \sigma_2 \rbrace$ for two distinct elements $\sigma_1, \sigma_2 \in \Sigma(G)$, and the whole set $\Sigma(\Gamma)$ is not admissible. Thus $\mathrm M_\Gamma$ turns out to have two irreducible components of dimension~$1$ meeting at the point~$X_0$.

Despite all the examples show the same geometrical picture of~$\mathrm M_\Gamma$, our motivation for constructing them was to reveal different combinatorial types of a non-admissible pair of spherical roots in~$\Sigma(\Gamma)$. Namely, as every admissible subset of $\Sigma(\Gamma)$ is a set of simple roots of a root system in $\QQ \Gamma$ (see \S\,\ref{subsec_comb_inv}), it is clear that the set $\lbrace \sigma_1, \sigma_2 \rbrace$ is automatically not admissible whenever $(\sigma_1, \sigma_2) > 0$; this happens in Example~\ref{exm_reducible1}. On the other hand, by Remark~\ref{rem_AP_consequences}(\ref{rem_AP_consequences_b}) the set $\lbrace \sigma_1, \sigma_2 \rbrace$ is automatically admissible if $\sigma_1, \sigma_2 \notin \Pi$. Our remaining four examples show that in the situation $\lbrace \sigma_1, \sigma_2 \rbrace \cap \Pi \ne \varnothing$ there are no simple conditions like $(\sigma_1, \sigma_2) = 0$ or $(\sigma_1, \sigma_2) < 0$ under which the set $\lbrace \sigma_1, \sigma_2 \rbrace$ is automatically admissible, and this holds regardless of whether both $\sigma_1, \sigma_2$ are simple roots or only one of them is simple. The following table demonstrates the combinatorial differences of our examples.

\begin{center}
\begin{tabular}{|c|c|c|c|c|c|}
\hline
Example no. & \ref{exm_reducible1} & \ref{exm_reducible2} & \ref{exm_reducible3} & \ref{exm_reducible4} & \ref{exm_reducible5} \\

\hline

Property &
\renewcommand{\tabcolsep}{0pt}%
\begin{tabular}{c}
$\sigma_1 {\notin} \Pi$, $\sigma_2 {\in} \Pi$,\\
$(\sigma_1, \sigma_2) {>} 0$
\end{tabular}
&
\renewcommand{\tabcolsep}{0pt}%
\begin{tabular}{c}
$\sigma_1, \sigma_2 {\in} \Pi$,\\
$(\sigma_1, \sigma_2) {=} 0$
\end{tabular}
&
\renewcommand{\tabcolsep}{0pt}%
\begin{tabular}{c}
$\sigma_1 {\notin} \Pi$, $\sigma_2 {\in} \Pi$,\\
$(\sigma_1, \sigma_2) {=} 0$
\end{tabular}
&
\renewcommand{\tabcolsep}{0pt}%
\begin{tabular}{c}
$\sigma_1, \sigma_2 {\in} \Pi$,\\
$(\sigma_1, \sigma_2) {<} 0$
\end{tabular}
&
\renewcommand{\tabcolsep}{0pt}%
\begin{tabular}{c}
$\sigma_1 {\notin} \Pi$, $\sigma_2 {\in} \Pi$,\\
$(\sigma_1, \sigma_2) {<} 0$
\end{tabular}
\\
\hline
\end{tabular}
\end{center}

\begin{example} \label{exm_reducible1}
Let $G=\SL_3$ and $\Gamma = \ZZ^+\lbrace 3\varpi_1, \varpi_1 + \varpi_2 \rbrace$. Then $\Gamma^\perp = \varnothing$ and $\ZZ \Gamma = \ZZ\lbrace \alpha_1, \alpha_2 \rbrace$. The spherical roots of $G$ compatible with the lattice $\ZZ \Gamma$ are $\alpha_1, \alpha_2$ and $\alpha_1 + \alpha_2$. Next, $\mathcal K^1 = \lbrace \varrho_1, \varrho_2 \rbrace$ where $\varrho_1 = (\alpha_1^\vee - \alpha_2^\vee)/3$ and $\varrho_2 = \alpha_2^\vee$. We have $\Sigma(\Gamma) = \lbrace \alpha_1 + \alpha_2, \alpha_1 \rbrace$ with $\rho(\mathcal D(\alpha_1)) = \lbrace \varrho_1, 2\varrho_1 + \varrho_2 \rbrace$. As $\langle 2\varrho_1 + \varrho_2, \alpha_1 + \alpha_2 \rangle = 1 > 0$, the set $\lbrace \alpha_1 + \alpha_2, \alpha_1 \rbrace$ is not admissible. Thus there are two maximal admissible subsets $\lbrace \alpha_1 + \alpha_2 \rbrace$ and $\lbrace \alpha_1 \rbrace$.
\end{example}

\begin{example} \label{exm_reducible2}
Let $G = \SL_2 \times \SL_2$, let $\varpi_i$ (resp.~$\alpha_i$) be the fundamental weight (resp. simple root) of the $i$th factor of~$G$, and consider the monoid $\Gamma = \ZZ^+ \lbrace 2\varpi_1, 2l\varpi_1 + 2\varpi_2 \rbrace$, where $l$ is a positive integer. Then $\Gamma^\perp = \varnothing$ and $\ZZ \Gamma = \ZZ \lbrace \alpha_1, \alpha_2 \rbrace$. The spherical roots of $G$ compatible with the lattice $\ZZ \Gamma$ are $\alpha_1$ and $\alpha_2$. Next, $\mathcal K^1 = \lbrace \varrho_1, \varrho_2 \rbrace$ where $\varrho_1 = (\alpha_1^\vee - l\alpha_2^\vee)/2$ and $\varrho_2 = \alpha_2^\vee/2$. We have $\Sigma(\Gamma) = \lbrace \alpha_1, \alpha_2 \rbrace$ with $\rho(\mathcal D(\alpha_1)) = \lbrace \varrho_1, \varrho_1 + 2l \varrho_2 \rbrace$ and $\rho(\mathcal D(\alpha_2)) = \lbrace \varrho_2 \rbrace$. If $l = 1$ then the set $\lbrace \alpha_1, \alpha_2 \rbrace$ is not admissible since $\langle \varrho_1 + 2\varrho_2, \alpha_2 \rangle = 1$ but $\varrho_1 + 2 \varrho_2 \notin \rho(\mathcal D(\alpha_2))$. If $l \ge 2$ then the set $\lbrace \alpha_1, \alpha_2 \rbrace$ is not admissible because $\langle \varrho_1 + 2l \varrho_2, \alpha_
2 \rangle = l > 1$. Thus there are two maximal admissible subsets $\lbrace \alpha_1 \rbrace$ and $\lbrace \alpha_2 \rbrace$.
\end{example}

\begin{remark}
For $l = 2$ we recover Luna's example mentioned in~\cite[Example~3.20]{AB2}.
\end{remark}

\begin{example} \label{exm_reducible3}
Let $G = \SL_2 \times G_0$, where $G_0$ is a connected semisimple algebraic group, and $\Gamma = \ZZ^+ \lbrace \alpha, l\alpha + \sigma \rbrace$, where $\alpha$ is the simple root of $\SL_2$, $\sigma$ is a dominant weight of $G_0$ such that $\sigma \in \Sigma(G_0) \setminus \Pi$, and $l$ is a positive integer. Then $\Gamma^\perp = \sigma^\perp \setminus \lbrace \alpha \rbrace$ and $\ZZ \Gamma = \ZZ\lbrace \alpha, \sigma \rbrace$. The spherical roots of $G$ compatible with the lattice $\ZZ \Gamma$ are $\sigma$ and~$\alpha$. Let $\varrho_0$ be the element of $\mathcal L$ such that $\langle \varrho_0, \alpha \rangle = 0$ and $\langle \varrho_0, \sigma \rangle = 2$. Then $\mathcal K^1 = \lbrace \varrho_1, \varrho_2 \rbrace$ where $\varrho_1 = (\alpha^\vee - l \varrho_0)/2$ and $\varrho_2 = \varrho_0/2$. We have $\Sigma(\Gamma) = \lbrace \sigma, \alpha \rbrace$ with $\rho(\mathcal D(\alpha)) = \lbrace \varrho_1, \varrho_1 + 2l\varrho_2 \rbrace$. As $\langle \varrho_1 + 2l\varrho_2, \sigma \rangle = l > 0$, the set $\lbrace \sigma, \alpha \rbrace$ is not admissible. Thus there are two maximal admissible subsets $\lbrace \sigma \rbrace$ and $\lbrace \alpha \rbrace$.
\end{example}

\begin{example} \label{exm_reducible4}
Let $G = \SL_4$ and $\Gamma = \ZZ^+ \lbrace 2\varpi_1 + (2l+1) \varpi_2, 2\varpi_2, \varpi_1 + \varpi_3 \rbrace$, where $l$ is a positive integer. Then $\Gamma^\perp = \varnothing$ and $\ZZ \Gamma = \ZZ\lbrace \alpha_1, \alpha_2, \alpha_3 \rbrace$. The spherical roots of $G$ compatible with the lattice $\ZZ \Gamma$ are $\alpha_1$, $\alpha_2$, $\alpha_3$, $\alpha_1 + \alpha_2$, and $\alpha_2 + \alpha_3$. Next, $\mathcal K^1 = \lbrace \varrho_1, \varrho_2, \varrho_3 \rbrace$ where $\varrho_1 = (\alpha_1^\vee - \alpha_3^\vee)/2$, $\varrho_2 = - (2l+1) \alpha_1^\vee/4 + \alpha_2^\vee/2 + (2l+1) \alpha_3^\vee/4$, and $\varrho_3 = \alpha_3^\vee$.
We have $\Sigma(\Gamma) = \lbrace \alpha_1, \alpha_2 \rbrace$ with $\rho(\mathcal D(\alpha_1)) = \lbrace \varrho_1, \varrho_1 + \varrho_3 \rbrace$ and $\rho(\mathcal D(\alpha_2)) = \lbrace \varrho_2, \varrho_2 + (2l+1) \varrho_1 \rbrace$. If $l = 1$ then the set $\lbrace \alpha_1, \alpha_2 \rbrace$ is not admissible since $\langle \varrho_2 + (2l+1) \varrho_1, \alpha_1 \rangle = 1$ but $\varrho_2 + (2l+1) \varrho_1 \notin \rho(\mathcal D(\alpha_1))$. If $l \ge 2$ then the set $\lbrace \alpha_1, \alpha_2 \rbrace$ is not admissible because $\langle \varrho_2 + (2l+1) \varrho_1, \alpha_1 \rangle = l > 1$. Thus there are two maximal admissible subsets $\lbrace \alpha_1 \rbrace$ and $\lbrace \alpha_2 \rbrace$.
\end{example}

\begin{example} \label{exm_reducible5}
Let $G = \SL_4$ and $\Gamma = \ZZ^+ \lbrace \varpi_1 + (2l + 1)\varpi_3, \omega_2, 2\omega_3 \rbrace$, where $l$ is a positive integer. Then $\Gamma^\perp = \varnothing$ and $\ZZ\Gamma = \ZZ \lbrace \alpha_1, \alpha_2, (\alpha_1 + \alpha_3)/2 \rbrace$. The spherical roots of $G$ compatible with the lattice $\ZZ \Gamma$ are $\alpha_1$, $\alpha_2$, $\alpha_3$, $\alpha_1 + \alpha_2$, and $\alpha_2 + \alpha_3$. Next, $\mathcal K^1 = \lbrace \varrho_1, \varrho_2, \varrho_3 \rbrace$ where $\varrho_1 = \alpha_1^\vee$, $\varrho_2 = \alpha_2^\vee$ and $\varrho_3 = (\alpha_3^\vee - (2l+1) \alpha_1^\vee)/2$. We have $\Sigma(\Gamma) = \lbrace \alpha_1 + \alpha_2, \alpha_3 \rbrace$ with $\rho(\mathcal D(\alpha_3)) = \lbrace \varrho_3, \varrho_3 + (2l+1) \varrho_1 \rbrace$. As $\langle \varrho_3 +(2l+1) \varrho_1, \alpha_1 + \alpha_2 \rangle = l > 0$, the set $\lbrace \alpha_1 +\alpha_2, \alpha_3 \rbrace$ is not admissible. Thus there are two maximal admissible subsets $\lbrace \alpha_1 + \alpha_2 \rbrace$ and $\lbrace \alpha_3 \rbrace$.
\end{example}

\begin{remark}
In Examples~\ref{exm_reducible2}, \ref{exm_reducible3}--\ref{exm_reducible5} the set $\Sigma(\Gamma)$ is admissible whenever $l = 0$.
\end{remark}

\begin{remark}
It would be interesting to construct examples of reducible moduli schemes $\mathrm M_\Gamma$ revealing other features. For instance, are there examples of $\mathrm M_\Gamma$ with arbitrarily many irreducible components and/or irreducible components of arbitrarily large dimension?
\end{remark}

\subsection{Examples where \texorpdfstring{$\mathrm M_\Gamma$}{M_Gamma} is a non-reduced point}
\label{subsec_examples_non-reduced}

In the examples below, the fact that $\mathrm M_\Gamma$ is a non-reduced point follows from $\Sigma(\Gamma) = \varnothing$ and $\Dev(\Gamma) \ne \varnothing$ thanks to Theorem~\ref{indexationofIrrComp} and Corollary~\ref{crl_reducedness_criterion}.

\begin{example}
Let $G = \SL_4$ and $\Gamma = \ZZ^+ \lbrace \lambda_1, \lambda_2, \lambda_3, \lambda_4 \rbrace$, where
\begin{align*}
\lambda_1 &= \varpi_2 + \varpi_3,\\
\lambda_2 &= 2\varpi_1 + 2\varpi_2 + 2\varpi_3, \\
\lambda_3 &= 2\varpi_1 + 2\varpi_2 + 3\varpi_3, \\
\lambda_4 &= 4\varpi_1 + 4\varpi_2 + 7\varpi_3.
\end{align*}
Then $\Gamma^\perp = \varnothing$ and $\ZZ \Gamma =  \ZZ\{2\varpi_1, \varpi_2, \varpi_3\}$. Note that $\lambda_3 = (\lambda_2 + \lambda_4)/3$ and $\Gamma$ is the intersection of the lattice $\ZZ \Gamma$ with the cone $\QQ^+\Gamma$, so that $\Gamma$ is saturated. The spherical roots of $G$ compatible with the lattice $\ZZ \Gamma$ are $\alpha_1$ and~$\alpha_3$. One easily checks that $\Sigma(\Gamma) = \varnothing$. Next, $\mathcal K^1 = \lbrace \varrho_1, \varrho_2, \varrho_3 \rbrace$ where
\begin{align*}
\varrho_1 &= 3\alpha_1^\vee/2 + 2\alpha_2^\vee - 2\alpha_3^\vee,\\
\varrho_2 &= - \alpha_2^\vee + \alpha_3^\vee, \\
\varrho_3 &= - \alpha_1^\vee + \alpha_2^\vee.
\end{align*}
We have $\langle \varrho_1, \alpha_1 \rangle = \langle \varrho_2, \alpha_1 \rangle = 1$, $\langle \varrho_3, \alpha_1 \rangle = -3$, and $\alpha_1^\vee = (2\varrho_1 + 4\varrho_2)/3$, which implies that $\alpha_1$ and $\varrho_1, \varrho _2$ satisfy conditions~(\ref{DR1})--(\ref{DR3}), whence $\alpha_1 \in \Dev(\Gamma)$. One easily checks that $\alpha_3 \notin \Dev(\Gamma)$ and so $\Dev(\Gamma) = \lbrace \alpha_1 \rbrace$. Therefore $\mathrm M_\Gamma$ is a non-reduced point.
\end{example}

\begin{example}
Let $G = \SL_2 \times \SL_2 \times \SL_2$, let $\varpi_i$ denote the fundamental weight of the $i$th factor of~$G$, and consider the monoid
\[
\Gamma = \ZZ^+ \lbrace 2\varpi_1, \varpi_2 + \varpi_3, 2p \varpi_1 + \varpi_2, 2q \varpi_1 + \varpi_3 \rbrace,
\]
where $\varpi_i$ stands for the fundamental weight of the $i$th factor of~$G$ and $p, q$ are positive integers. Then $\Gamma^\perp = \varnothing$ and $\ZZ \Gamma = \ZZ \lbrace 2\varpi_1, \varpi_2, \varpi_3 \rbrace$. Note that $\Gamma$ is the intersection of the lattice $\ZZ \Gamma$ with the cone $\QQ^+ \Gamma$, so that $\Gamma$ is saturated. The only spherical root of $G$ compatible with the lattice $\ZZ \Gamma$ is~$\alpha_1$. Next, $\mathcal K^1 = \lbrace \varrho_1, \varrho_2, \varrho_3, \varrho_4 \rbrace$ where
\begin{align*}
\varrho_1 &= \alpha_1^\vee/2 - p \alpha_2^\vee + p \alpha_3^\vee,\\
\varrho_2 &= \alpha_1^\vee/2 + q \alpha_2^\vee - q \alpha_3^\vee, \\
\varrho_3 &= \alpha_2^\vee,\\
\varrho_4 &= \alpha_3^\vee.
\end{align*}
Clearly, $\alpha^\vee = \frac{2q}{p+q} \varrho_1 + \frac{2p}{p+q} \varrho_2$. Then it is easy to see that  $\alpha_1 \in \Sigma(\Gamma)$ when $p = q$ and $\alpha_1 \in \Dev(\Gamma)$ when $p \ne q$. It follows that $\mathrm M_\Gamma$ is an affine line when $p = q$ and a non-reduced point when $p \ne q$.
\end{example}

\begin{remark}
It would be interesting to construct examples of non-reduced moduli schemes $\mathrm M_\Gamma$ revealing other features. For instance, are there examples of reducible and non-reduced~$\mathrm M_\Gamma$, examples of irreducible non-reduced $\mathrm M_\Gamma$ of arbitrarily large dimension, examples where $\mathrm M_\Gamma$ is a non-reduced point with tangent space of arbitrarily large dimension?
\end{remark}


\end{document}